\documentclass[reqno,12pt]{amsart}

\usepackage{color} % black,white,red,green,blue,cyan,magenta,yellow
\usepackage[pdftex]{graphicx}
\usepackage[pdftex]{hyperref}
 \hypersetup{
        unicode=false,          % non-Latin characters in Acrobat’s bookmarks
        pdftoolbar=true,        % show Acrobat’s toolbar?
        pdfmenubar=true,        % show Acrobat’s menu?
        pdffitwindow=false,     % window fit to page when opened
        pdfstartview={FitH},    % fits the width of the page to the window
        pdftitle={Manuscript},      % title
        pdfauthor={Yunrong Zhu},   % author
        pdfsubject={Mathematics},    % subject of the document
        pdfcreator={Yunrong Zhu},  % creator of the document
        pdfproducer={Yunrong Zhu}, % producer of the document
        pdfkeywords={PDE, Finite Element, Multigrid, Preconditioner}, % list of keywords
        pdfnewwindow=true,      % links in new window
        colorlinks=true,        % false: boxed links; true: colored links
        linkcolor=red,          % color of internal links
        citecolor=blue,         % color of links to bibliography
        filecolor=magenta,      % color of file links
        urlcolor=cyan           % color of external links
    }

\usepackage{times}
\usepackage{amsfonts}

\usepackage{amsmath}
\usepackage{amsthm}
\usepackage{amssymb}
\usepackage{amsbsy}
\usepackage{amscd}

\newtheorem{theorem}{Theorem}[section]
\newtheorem{corollary}[theorem]{Corollary}
\newtheorem{lemma}[theorem]{Lemma}

\newtheorem{definition}[theorem]{Definition}

\newtheorem{remark}[theorem]{Remark}
\newtheorem{alg}{Algorithm}[section]

\numberwithin{equation}{section}  %amsmath command: tie counter to section

\newcommand{\cJ}{{\mathcal J}}      %Jump
\newcommand{\cT}{{\mathcal T}}      %Triangulation

\begin{document}

\title[MG for Elliptic Problems with Discontinuous Coefficients]
{Multilevel Preconditioners for Reaction-Diffusion Problems with Discontinuous Coefficients}

\author[T. Kolev]{Tzanio V. Kolev}
\email{kolev1@llnl.gov}
\address{Center for Applied Scientific Computing, \\
Lawrence Livermore National Laboratory, \\
P.O. Box 808, L-561, Livermore, CA 94551, USA}

\author[J. Xu]{Jinchao Xu}
\email{xu@math.psu.edu}
\address{Department of Mathematics\\
         Penn. State University\\
         University Park, PA 16802, USA}

\author[Y. Zhu]{Yunrong Zhu}
\email{zhuyunr@isu.edu}
\address{Department of Mathematics\\
         Idaho State University\\
         Pocatello, ID 83209-8085, USA}

\subjclass{65F10, 65N20, 65N30}

\thanks{This work performed under the auspices of the U.S. Department of Energy
  by Lawrence Livermore National Laboratory under Contract DE-AC52-07NA27344,
  LLNL-JRNL-663816.  Jinchao Xu was supported in part by NSF DMS 1217142 and DOE Award
  \#DE-SC0009249.  Yunrong Zhu was supported in part by NSF DMS 1319110, and in part by
  University Research Committee Grant No. F119 at Idaho State University,
  Pocatello, Idaho. }
%\thanks{SC were supported in part by NSF Award~0715146.}

\date{\today}

\keywords{Reaction-Diffusion Equations, Multigrid, BPX, Discontinuous Coefficients, Robust Solver, Multilevel Preconditioners}

\begin{abstract}
In this paper, we extend some of the multilevel convergence results obtained by Xu and Zhu in [Xu and Zhu, M3AS 2008], to the case of second order linear reaction-diffusion equations. Specifically, we consider the multilevel preconditioners for solving the linear systems arising from the linear finite element approximation of the problem, where both diffusion and reaction coefficients are piecewise-constant functions. We discuss in detail the influence of both the discontinuous reaction and diffusion coefficients to the performance of  the classical BPX and multigrid V-cycle preconditioner. 
%We show that for constant diffusion coefficient, the classical BPX or multigrid V-cycle preconditioner is nearly optimal, independent of the jumps in the reaction coefficient. On the other hand, in case that the reaction coefficient is a constant, we recover the results from \cite{Xu.J;Zhu.Y2008}.
\end{abstract}

\maketitle

%\clearpage

%\vspace*{-1.2cm}
%%{\footnotesize
%\tableofcontents
%%}
%\vspace*{-0.7cm}
%
%\clearpage
%

%%%%%%%%%%%%%%%%%%%%%%%%%%%%%%%%%%%%%%%%%%%%%%%%%%%%%%%%%%%%%%%%%%%%%%%%%%%%%%%%
%  Introduction
%%%%%%%%%%%%%%%%%%%%%%%%%%%%%%%%%%%%%%%%%%%%%%%%%%%%%%%%%%%%%%%%%%%%%%%%%%%%%%%%
\section{Introduction}
In this paper, we will discuss the convergence of multilevel preconditioners for the linear finite element approximation
of the second order elliptic boundary value problem with discontinuous coefficients:
\begin{equation}\label{eq:model}
    \left\{\begin{array}{rl}
      -\nabla\cdot(\omega\nabla u) + \rho \, u=f &\mbox{ in } \Omega,\\
      u=0 & \mbox{ on } \Gamma_D, \\
      \omega\frac{\partial u}{\partial n}=g_N & \mbox{ on } \Gamma_N
      \end{array}
    \right.
\end{equation}
where $\Omega\in \mathbb{R}^d (d= 2\mbox{ or } 3)$ is a polygonal or
polyhedral domain with Dirichlet boundary $\Gamma_D$ and Neumann
boundary $\Gamma_N.$
While such problems arise in a wide variety of practical applications, our interest in \eqref{eq:model}
is motivated by the subspace problems in auxiliary-space preconditioners for the definite Maxwell
equations \cite{hx,ams_jcm}.

Multigrid algorithms are a family of powerful solution techniques
which are frequently applied to the finite element discretizations of
\eqref{eq:model}. When the coefficients $\omega > 0$ and $\rho \geq 0$ are constant, it is well known that Multigrid is an efficient optimal solver; while its additive version, the BPX algorithm, is an optimal preconditioner (see for example \cite{Bramble.J;Pasciak.J;Xu.J1990,Hackbusch.W1985}.)
In many practical applications, however, the coefficients  $\omega$ and $\rho$ of
\eqref{eq:model} describe material properties, which can be
considered constant in the material subdomains, but may have large jumps
on the material interfaces.

There have been a lot of works devoted to developing efficient iterative solvers for solving the finite element discretization of \eqref{eq:model}, which are robust with respect to the jumps in the diffusion coefficient $\omega$ (when $\rho \equiv 0$),  (see \cite{Bramble.J;Xu.J1991,Wang.J1994,Wang.J;Xie.R1994,Dryja.M;Sarkis.M;Widlund.O1996,Oswald.P1999c,Graham.I;Hagger.M1999} for examples). For general cases, one usually need some special techniques to obtain robust iterative methods, (cf. \cite{Chan.T;Wan.W2000,Scheichl.R;Vainikko.E2007,Graham.I;Lechner.P;Scheichl.R2007,Aksoylu.B;Graham.I;Klie.H;Scheichl.R2008}). Recently, Xu and Zhu addressed in \cite{Xu.J;Zhu.Y2008,Zhu.Y2008} the performance of the BPX and Multigrid V-cycle preconditioners for \eqref{eq:model} in the case of discontinuous $\omega$ and  $\rho=0$. It was shown that the jumps in $\omega$ affect only a small number of eigenvalues, and therefore the (asymptotic) convergence rate of the preconditioned conjugate gradient (PCG) method is uniform with respect to the jumps and the mesh size.
See also \cite{Galvis.J;Efendiev.Y2010a,Scheichl.R;Vassilevski.P;Zikatanov.L2012,Ayuso-de-Dios.B;Holst.M;Zhu.Y;Zikatanov.L2014,Chen.L;Holst.M;Xu.J;Zhu.Y2012,Zhu.Y2014}
and the references cited therein for further developments in different directions.

All the analysis mentioned above focused on pure diffusion equation, and very little attention has been paid for the case when $\rho$ is nonzero. In many applications, such as time discretization of heat conduction in composite materials, the equation involves a lower order term.  In the case that $\omega_{i} =\rho_{i}$, a robust overlapping domain decomposition  method was developed for a two-dimensional model problem in \cite{Cho.S;Nepomnyaschikh.S;Park.E2005}. Recently, the paper \cite{Kraus.J;Wolfmayr.M2013} discussed the performance of the algebraic multilevel iteration (AMLI) methods for the finite element discretizations for \eqref{eq:model} in 2D, which are based on a multilevel block factorization and polynomial stabilization.

In this paper, we study the performance of the classical multilevel preconditioners (BPX and multigrid V-cycle) on the finite element discretization of equation \eqref{eq:model}, with emphasis on the discussion of the influence of both the discontinuous  reaction and diffusion coefficients on the convergence of these multilevel preconditioners. We classify the coefficients in two different cases. In the first case, we require that both $\omega$ and $\rho$ have the same coefficient distribution, namely, if $\omega_{i} \ge \omega_{j}$ then $\rho_{i} \ge \rho_{j}$ and vice versa. Note that this includes the case when $\rho$ is a global constant. In this case, we recover the results from \cite{Xu.J;Zhu.Y2008}.  On the other hand, when $\omega$ and $\rho$ have different distributions, it seems that the performance of the preconditioners deteriorate with the jumps (see the numerical examples in Section~{\ref{sec:num3}}). In this case, we showed that the convergence rate depends on the minimal of the jumps in $\omega$ and $\rho$. As a special case, when $\omega$ is a global constant, or only varies moderately in the whole domain, we show that the multilevel preconditioners are robust with respect to both coefficients, and the mesh size.

%The main difficulties of analyzing the multilevel preconditioners for the jump coefficients problem is the lack of the stable decomposition in the weighted norms. In particular, we need both $\omega$-weighted $L^{2}$ approximation, and $\rho$-weighted $L^{2}$ stability  properties to define a stable decomposition.

The remainder of the paper is organized as follows: in the next section we
investigate the Jacobi and Gauss-Seidel preconditioners.
Then, in Section \ref{sec-interp}, we consider an interpolation
operator which is needed in the analysis of the BPX and Multigrid algorithms, carried out
in the following Section \ref{sec:multilevel}.
The developed theory is illustrated by several numerical experiments
collected in Section \ref{sec-numerical}.

Throughout the paper we use the standard notation for Sobolev spaces and their norms.
We will use the notation $x_1\lesssim y_1$, and $x_2\gtrsim y_2$,
whenever there exist constants $C_1, C_2$ independent of the mesh size
$h$ and the coefficients $\omega$ and $\rho$,
and such that $x_1 \le C_1 y_1$
and $x_2\ge C_2 y_2$, respectively. We also use the notation $x\simeq y$ for $C_{1} x \le y\le C_{2} x$.

%%%%%%%%%%%%%%%%%%%%%%%%%%%%%%%%%%%%%%%%%%%%%%%%%%%%%%%%%%%%%%%%%%%%%%%%%%%%%%%%
%  Jacobi and Gauss-Seidel Preconditioning
%%%%%%%%%%%%%%%%%%%%%%%%%%%%%%%%%%%%%%%%%%%%%%%%%%%%%%%%%%%%%%%%%%%%%%%%%%%%%%%%
\section{Notation and Preliminaries}\label{sec:pre}
 In this section, we establish the notations and review a few preliminary tools that will be needed for the subsequent analysis, following those in \cite{Xu.J;Zhu.Y2008}. We consider solving the model equation \eqref{eq:model} in a polyhedral domain $\Omega \subset \mathbb{R}^{d}$ for $d=2$ or 3, and assume that there is a set of non-overlapping subdomains $\{\Omega_m\}_{m=1}^M$ such that $\overline{\Omega}=\cup_{m=1}^{M} \overline{\Omega}_{m}$, on which the diffusion coefficient $\omega(x)$ and the reaction coefficient $\rho(x)$ are constants, denoted by $\omega_{m} :=\omega(x)|_{\Omega_{m}}$ and $\rho_{m}:=\rho(x)|_{\Omega_{m}}$ for each $m = 1, \cdots, M$ respectively.

Let $V=H_{D}^{1}(\Omega)$ be the space that consists of $H^{1}(\Omega)$ functions with vanishing trace on the Dirichlet boundary $\Gamma_D \subset \partial \Omega$. The variational problem of \eqref{eq:model} reads: Given $f\in L^{2}(\Omega)$ and $g_{N} \in H^{1/2}(\Gamma_{N})$, find $u\in V$ such that
$$
        a(u,v) = \int_{\Omega} fv \;dx + \int_{\Gamma_{N}} g_{N} v\; ds, \qquad \forall v\in H_{D}^{1}(\Omega),
$$
where the bilinear form $a(\cdot,\cdot)$ is given by
$$
a(u,v) = \sum_{m=1}^{M}  \int_{\Omega_{m}} \omega_m\nabla u\cdot \nabla v\; dx
+  \sum_{m=1}^{M}\int_{\Omega_{m}} \rho_m u v\; dx\,.
$$
For the analysis of this paper, we will need the following weighted semi-norms and norms. Given any piecewise-constant
coefficient $\tau=\{\tau_{1}, \cdots, \tau_{M}\}>0$, we define weighted $L^{2}$ norm and $H^1$ semi-norm
by
$$
\|u\|_{0,\tau}^2 = (u,u)_{0,\tau} = \sum_{m=1}^{M} \tau_m \|u\|^2_{L^{2}(\Omega_m)} \,,
\qquad\text{and}\qquad
|u|_{1,\tau}^2 = \sum_{m=1}^{M} \tau_m |u|^2_{H^1(\Omega_m)} \,.
$$
In this notation, the bilinear form of interest is
$a(u,u) = |u|_{1,\omega}^2 + \|u\|_{0,\rho}^2$. With a little abuse of the notation, we will use the same notation for the case when  $\rho =0$ in some subdomains of $\Omega$, although in this case $\|\cdot\|_{0,\rho}$ is not a norm.

%  As in \cite{Xu.J;Zhu.Y2008}, we introduce the subspace
%$$
%	\widetilde{H}_{D}^{1}(\Omega):= \left\{ v \in H_{D}^{1} \>:\> \int_{\Omega_m} v dx= 0 \,, m \in I \right\} \,,
%$$
%and $ \tilde{V}_h = V_h \cap \widetilde{H}_{D}^{1},$
%where $I$ stands for the set of subdomains which boundary does not intersect $\Gamma_D$.
%Note that the Poincar\'e-Friedrichs inequality
%\begin{equation} \label{poincare}
%\|v\|_{L^{2}(\Omega_m)} \lesssim |v|_ {H^1(\Omega_m)}
%\end{equation}
%holds on $\widetilde{H}_D^{1}(\Omega)$.

Let $\cT_{h}$ be a quasi-uniform triangulation of $\Omega$. We
assume that all $\Omega_m$ are of unit size, and that their geometries
are resolved exactly by the triangulation. Let $V_h \subset V$ be
the corresponding linear Lagrangian finite element spaces. Then the finite element discretization of \eqref{eq:model} reads: find $u_{h}\in V_h,$ such
that
$$
        a(u_{h},v)=\int_{\Omega}f v \; dx + \int_{\Gamma_N} g_N v \; ds,\qquad \forall v\in V_h.
$$
We define a linear symmetric positive definite (SPD) operator $A: V_h\to
V_h$ by
$$(Aw, v)=(w,v)_{A}=a(w,v), \qquad \forall v, w\in V_{h}.$$
and use the notation $\|\cdot\|_{A} =\sqrt{a(\cdot,\cdot)}$ to denote the energy norm. We need to solve the following operator equation,
\begin{equation}\label{eq:eq}
   Au=b,
\end{equation}
where $b$ is defined as $\langle b,v\rangle=\int_{\Omega}f v \; dx+\int_{\Gamma_N} g_N v \; ds,\;\; \forall v\in V_h.$ Given the nodal basis $\{\phi_{i}\}_{i=1}^{N}$ of $V_{h}$, let $\mathbb{A} = (a_{ij})_{N\times N}$  with $a_{ij} = a(\phi_{j}, \phi_{i})$ for $i,j = 1,\cdots, N$ be the matrix representation of $A$.

Since $A$ is SPD, by the classical PCG theory we know that the convergence rate of the iterative
method for $A$ with a preconditioner, say $B$,  is determined, by the (generalized) condition number of the preconditioned system: $\kappa(BA) := \lambda_{N}(BA) / \lambda_{1}(BA)$, where $\lambda_{i}(BA)$ for $i=1,\cdots, N$ are the eigenvalues of $BA$ satisfying $\lambda_{1} \le \lambda_{2} \le \cdots \le \lambda_{N}$. However, if we know a priori that the spectrum $\sigma(BA)$ of $BA$ satisfies $\sigma(BA)=\sigma_0(BA)\cup \sigma_1(BA)$, where
$\sigma_0(BA)=\{\lambda_1,\ldots ,\lambda_{m}\}$ contains
all extreme (``bad'') eigenvalues,
and the remaining eigenvalues contained in
$\sigma_1(BA)=\{\lambda_{m+1}, \ldots , \lambda_{N}\}$ are bounded from above and below, i.e., $\lambda_j \in [\alpha,\beta]$ for $j=m+1,\ldots, N$, then the error
at the $k$-th iteration of the PCG algorithm can be  bounded by
(cf. e.g.~\cite{Axelsson.O1994,Hackbusch.W1994,Axelsson.O2003}):
\begin{equation}
\label{eqn:CG}
        \frac{\|u-u_k\|_{A}} {\|u-u_0\|_{A}}\le 2(\kappa(BA)-1)^{m}
        \left(\frac{\sqrt{\beta/\alpha}-1}{\sqrt{\beta/\alpha}+1}\right)^{k-{m}}\;.
\end{equation}
Specifically, if the number of extreme eigenvalues $m$ is small, then the \emph{asymptotic} convergence rate of the resulting PCG method will be determined by the ratio  $(\beta/\alpha)$, which is the so-called \emph{effective condition number} (cf. \cite{Nabben.R;Vuik.C2004,Xu.J;Zhu.Y2008}).
 \begin{definition}
\label{def:effcond}
Let $T:V_{h}\to V_{h}$ be a symmetric positive definite linear
operator, with eigenvalues $0 < \lambda_{1} \le \cdots \le
\lambda_{N}$. For $m = 0, 1, \cdots N-1$, the \emph{$m$-th effective condition number} of
$T$ is defined by
    \begin{equation*}
         \kappa_{m}(T) := \frac{\lambda_{N}(T)}{\lambda_{m+1}(T)}.
         \end{equation*}
\end{definition}
\begin{remark}
\label{rk:minmax}
To estimate the effective condition number, and specifically $\lambda_{m+1}(T)$, a standard tool is the \emph{min-max principle} (see e.g.~\cite[Theorem 8.1.2]{Golub.G;Van-Loan.C1996}):
  $$
        \lambda_{m+1}(T)=\max_{\dim (S)=m}\min_{0\neq v\in S^{\perp}}\frac{(Tv,v)}{(v,v)}
  $$
In particular, for any subspace $\widetilde{V}\subset V_{h}$ with ${\rm
  dim}(\widetilde{V})=n-m$, it holds
$$
    \lambda_{m+1}(T)\ge \min_{0\neq v\in \widetilde{V}}\frac{(Tv,v)}{(v,v)}.
$$
\end{remark}
%If both $A$ and $B$ are SPD operators, then $BA$ is SPD in the inner product induced by $B^{-1}$ and $A$. Below, we shall apply Theorem \ref{th:minimax} to $T = BA$ and $(u,v)_{V}:=(B^{-1}u,v)_{L^2}$. Therefore if we have an inequality of the type $(Av,v)\geq c(B^{-1}v,v)$ for all $v$ in a suitable subspace $V_0$ with ${\rm dim}(V_0)=n-m$, we can get a lower bound of $\lambda _{m+1}(BA)$.

As in \cite{Xu.J;Zhu.Y2008}, we introduce a subspace $\widetilde{H}_{D}^{1}(\Omega) \subset H_{D}^{1}(\Omega)$ as:
$$\widetilde{H}_D^{1}(\Omega)=\left\{v\in H_D^{1}(\Omega): \int_{\Omega_m} v=0, \;\; m\in I\right\},$$
where $I$ is the index set of all subdomains not touching the Dirichlet boundary $\Gamma_{D}$, defined as $I:=\left\{m: {\rm meas}\left(\partial\Omega_m\cap \Gamma_D\right)=0\right\}$ where ${\rm meas}(\cdot)$ is the $d-1$ measure. The subdomains indexed by $I$, are sometimes called the \emph{floating subdomains}. Similarly, we define the finite element subspace $\widetilde{V}_{h}:=V_{h} \cap \widetilde{H}_{D}^{1}(\Omega)$. It is obvious that if $m_{0}$ is the cardinality of the index set $I$, then $\mbox{dim}(\widetilde{V}_h)=N-m_0.$ Moreover, the following Poincar\'e-Friedrichs inequality holds:
  \begin{equation}\label{eq:poincare}
    c_0 \left\|v\right\|_{0,\omega}\le \left\|\nabla v\right\|_{0,\omega},\;\; \forall v\in \widetilde{H}_D^{1}(\Omega).
  \end{equation}
We shall emphasize that $m_0$ is a fixed number which depends only on the distribution
of the coefficient $\omega$ on the domain.

We conclude this section by a discussion on the  simple (but commonly used) Jacobi and Gauss-Seidel preconditioners.
Note that the jumps in $\rho$ do not influence the condition number estimates.

\begin{theorem}[cf. Theorem 2.2 in \cite{Xu.J;Zhu.Y2008}] \label{th-gs}
Let $\mathbb{A}$ be the stiffness matrix corresponding to $a(\cdot,\cdot)$
in $V_h$, and let $\mathbb{D}$ be its diagonal.
The condition number of $\mathbb{D}^{-1} \mathbb{A}$ (Jacobi preconditioning)
depends on the mesh size and the coefficient $\omega$:
$$
\kappa (\mathbb{D}^{-1} \mathbb{A}) \lesssim h^{-2} \mathcal{J}(\omega) \,,
$$
where
\begin{equation}\label{eq:Jminmax}
\mathcal{J}(\omega) = \frac{\max_m \omega_m}{\min_m \omega_m} \,.
\end{equation}
On the other hand, the $m_{0}$-th effective condition number is independent of the
coefficients $\omega$ and $\rho$:
$$
\kappa_{m_0} (\mathbb{D}^{-1} \mathbb{A}) \lesssim h^{-2} \,.
$$
Here $m_0 = |I|$, is the number of interior subdomains $\Omega_m$.
\end{theorem}
\begin{proof}
For ease of presentation, we introduce a mesh dependent coefficient defined as
$$
        \omega_{h} = \omega + h^{2}\rho.
$$
Note that when $\rho = 0$, we have $\omega_h = \omega$ as in \cite{Xu.J;Zhu.Y2008}.
Given any $v_{h} \in V_h$, let $\mathbf{v}$ be its vector representation in the nodal basis
of $V_h$.

First of all, by inverse inequality
$$
\mathbf{v}^{t} \mathbb{A} \mathbf{v} = a(v_{h},v_{h}) \lesssim h^{-2} \|v_{h}\|^2_{0,\omega_h}
\simeq \mathbf{v}^{t} \mathbb{D}\, \mathbf{v} \,.
$$
%where $\omega_{h}$ is the mesh dependent coefficient  defined as:
%\begin{equation} \label{omegah}
%\omega_h = \omega + h^{2} \rho \,.
%\end{equation}
This inequality implies that   $\lambda_{\max} (\mathbb{D}^{-1}\mathbb{A}) \lesssim 1.$
On the other hand, by Poincar\'e-Friedrichs inequality, we have
$$
        \|v_{h}\|^{2}_{0, \omega} \le \max_{m}\omega_{m} \|v_{h}\|^{2}_{L^{2}(\Omega)} \lesssim \max_{m}\omega_{m} \|\nabla v_{h} \|^{2}_{L^{2}(\Omega)} \lesssim \cJ(\omega)|v_{h}|^{2}_{1,\omega}.
$$
Since $ h^2 \lesssim 1$, we obtain
$$
a(v_{h},v_{h}) \gtrsim \mathcal{J}(\omega)^{-1} \|v_{h}\|^2_{0,\omega} + \|v_{h}\|^2_{0,\rho}
\gtrsim \mathcal{J}(\omega)^{-1} \|v_{h}\|^2_{0,\omega_h} \,.
$$
which implies $ h^{2} \mathcal{J}(\omega)^{-1} \mathbf{v}^{t} \mathbb{D}\, \mathbf{v}
\lesssim \mathbf{v}^{t} \mathbb{A}\, \mathbf{v}$. Thus the minimum eigenvalue of $\mathbb{D}^{-1} \mathbb{A}$ is bounded by  $\lambda_{\min} (\mathbb{D}^{-1} \mathbb{A}) \gtrsim h^{2} \cJ^{-1}(\omega).$ This proves  $\kappa(\mathbb{D}^{-1} \mathbb{A})\lesssim h^{-2} \cJ(\omega).$

Finally, when $v_{h} \in \widetilde{V}_h$ the Poincar\'e-Friedrichs inequality \eqref{eq:poincare} implies
$$
        \|v_{h}\|^2_{0,\omega_h} \lesssim a(v_{h},v_{h}) \,,
$$
and therefore $h^{2}\, \mathbf{v}^{t} \mathbb{D}\, \mathbf{v} \lesssim
\mathbf{v}^{t} \mathbb{A}\, \mathbf{v}$. Then by the min-max principle (cf. Remark~\ref{rk:minmax}), we obtain that
$\lambda_{m_{0}+1}(\mathbb{D}^{-1} \mathbb{A}) \gtrsim h^{-2},$
since ${\rm dim}(\widetilde{V}_{h}) = {\rm dim}(V_{h}) -m_{0}.$ Thus we obtain the desired estimate for the $m_{0}$-th effective condition number $\kappa_{m_{0}}(\mathbb{D}^{-1} \mathbb{A})$.
\end{proof}

\begin{remark}
Analogous results hold for symmetric Gauss-Seidel preconditioner based on certain spectral equivalence between Jacobi and Gauss-Seidel iterations for SPD matrices (see \cite{Vassilevski.P2008} for more details).
\end{remark}

%%%%%%%%%%%%%%%%%%%%%%%%%%%%%%%%%%%%%%%%%%%%%%%%%%%%%%%%%%%%%%%%%%%%%%%%%%%%%%%%
%  Interpolation operator
%%%%%%%%%%%%%%%%%%%%%%%%%%%%%%%%%%%%%%%%%%%%%%%%%%%%%%%%%%%%%%%%%%%%%%%%%%%%%%%%
\section{Interpolation operator} \label{sec-interp}
The analysis of the multilevel preconditioner relies on the approximation and stability of certain interpolation operator. In this section we describe the dual basis-based interpolation
operator from \cite{Scott.R;Zhang.S1990}, and show how it can be used to derive
simultaneous estimates in two different weighted inner products.

Let $T \in \cT_h$ be a fixed mesh element and $\{\lambda_{T,i}\}$ be the set of
its linear finite element shape functions. The local mass matrix on $T$ has entries
$$
(\mathbb{M}_T)_{ij} = \int_T \lambda_{T,j}\, \lambda_{T,i}  dx\,,
$$
and it is easy to check that $\mathbb{M}_T$ is spectrally equivalent to a diagonal matrix:
$\mathbb{M}_T \eqsim h^d \mathbb{I}_T$.
Let $\{\mu_{T,i}\}$ be the dual basis of $\{\lambda_{T,i}\}$, i.e.
$\mu_{T,i} = \sum_j \alpha_{ij} \lambda_{T,j}$ is such that
$$
\int_T \lambda_{T,j}\, \mu_{T,i} dx= \delta_{ij}\,.
$$
Then
\begin{equation} \label{eq-mu2}
\int_T \mu^2_{T,i} dx= \alpha_{ii} = (\mathbb{M}_T^{-1})_{ii} \eqsim h^{-d} \,.
\end{equation}

Given $v \in L^{2}(\Omega)$ we define $\Pi_h v \in V_h$ by specifying its values
in the vertices of $\cT_h$. Specifically, the value at a vertex $x$ is determined
using an associated element $T_x \in \cT_h$:
\begin{equation}\label{eqn:pih}
(\Pi_h v)(x) = \int_{T_x} v\, \mu_{x} dx\,,
\end{equation}
and $(\Pi_h v)(x) = 0$ if $x\in \Gamma_{D}$.
Here $\mu_{x}:=\mu_{T_{x},i}$ is the dual basis at $x$ in $T_x$, where $i$ is the index of $x$ as a vertex of $T_{x}$.

\begin{remark}
Let $Q_{T_x}$ be the local $L^{2}$-projection on $T_x$, i.e. $Q_{T_x}v$ is the
unique linear combination of $\{\lambda_{T_x,i}\}$ which satisfies
$$
(Q_{T_x} v, w)_{L^{2}(T_x)} = (v, w)_{L^{2}(T_x)}
$$
for any $w \in \text{span}\,\{\lambda_{T_x,i}\}$.
Then $\Pi_h$ can be equivalently defined by
$$
(\Pi_h v)(x) = (Q_{T_x} v)(x).
$$
\end{remark}

We remark that the choice of $T_{x}$ is not unique. Given a particular ordering of the subdomains, say $\Omega_{1}, \cdots, \Omega_{M}$, we choose $T_{x} \subset \Omega_{k}$ where $k$ is the minimal index of all the subdomains that contain $x$. Note that this ordering has nothing to do with the actual geometry distribution of the coefficients.
In order to make $\Pi_{h}$ satisfy certain stability property in the weighted norms, we may label the subdomains such that
$
        \omega_{1} \ge \omega_{2} \ge \cdots \ge \omega_{M},
$
or
$ \rho_{1}\ge \rho_{2} \ge \cdots \ge \rho_{M}.$
In this case, the choice of $T_{x}$ guarantees that the coefficient in $T_{x}$ is the maximum of all the coefficients in the neighborhood of $x$. By a standard argument, we have the following result on $\Pi_{h}$.
%It is well-known that such defined $\Pi_{h}$ is $L^{2}$-stable (cf. \cite{Scott.R;Zhang.S1990}), that is $$\|\Pi_{h} v\|_{L^{2}(\Omega)} \lesssim \|v\|_{L^{2}(\Omega)}.$$
%In order to make $\Pi_{h}$ to satisfy the desired stability property in the weighted norms,
%This choice of $T_{x}$ guarantees that the coefficients in $T_{x}$ are the maximum of all the coefficients in the macro-element of $x$.

\begin{lemma}
\label{lm:l2stable}
The interpolation operator $\Pi_h: L^{2}(\Omega) \to V_{h}$ defined above satisfies that
\begin{eqnarray}
\|\Pi_h v\|_{L^{2}(\Omega_m)}^2 &\lesssim& \sum_{k \leq m} \|v\|_{L^{2}(\Omega_k)}^2, \; \forall v\in L^{2}(\Omega)\label{eq-pih-l2stable}
%\\
%|\Pi_h v|_{H^{1}(\Omega_m)}^2 &\lesssim& \sum_{k \leq m} |v|_{H^{1}(\Omega_k)}^2, \; \forall v\in H_{D}^{1}(\Omega). \label{eq-pih-h1stable}
\end{eqnarray}
\end{lemma}
\begin{proof}
        For each element $T\subset \Omega_{m}$, we estimate $\|\Pi_{h} v\|_{L^{2}(T)}$ as follows:
        \begin{eqnarray*}
                \|\Pi_{h}v\|_{L^{2}(T)} &\le& \sum_{i=1}^{d+1} |\Pi_{h}v(x_{i})| \|\phi_{i}\|_{L^{2}(T)}
                \lesssim h^{d/2} \sum_{i=1}^{d+1}\int_{T_{x_{i}}}\mu_{x_{i}} v dx\\
                &\lesssim& h^{d/2} \sum_{i=1}^{d+1}\|\mu_{x_{i}}\|_{L^{2}(T_{x_{i}})} \|v \|_{L^{2}(T_{x_{i}})}
                \lesssim \|v\|_{L^{2}(S_T)},
        \end{eqnarray*}
        where $S_{T} = \bigcup\{T_{x_{i}} \in \cT_{h}: T_{x_{i}} \mbox{ is the element associated with the vertex } x_{i}\}.$ In the last step, we used the property \eqref{eq-mu2} on $\mu_{x_{i}}$. Notice that $T_{x_{i}} \subset \Omega_{k}$ where $k$ is the minimal index of all the subdomains that intersect at $x_{i}$. Therefore, the $L^{2}$ stability of $\Pi_{h}$  \eqref{eq-pih-l2stable} follows by summing up all the elements in $\Omega_{m}$ on both sides.
%First, observe that using Schwarz inequality and \eqref{eq-mu2} we get
%$$
%(\Pi_h v)^2(x)
%%% \leq \int_{T_x} v^2 \int_{T_x} \mu^2_{T_x,i}
%\lesssim h^{-d} \|v\|_{L^{2}(T_x)}^2 \,.
%$$
%Therefore
%\begin{equation} \label{eq-pih-l2stable-T}
%\|\Pi_h v\|_{L^{2}(T)}^2 \lesssim  \|v\|_{L^{2}(\Pi_h(T))}^2 \,,
%\end{equation}
%where $\Pi_h(T)$ is the union of all elements that are associated with the vertices of $T$.
%
%The decomposition $\{\Omega_m\}_{m=1}^M$ induces a non-overlapping
%splitting of the vertices in the mesh, where each vertex is associated
%with the minimal-index domain that contains him.  Define $\Pi_h$ by
%associating each vertex with an element in its corresponding
%subdomain.  Then \eqref{eq-pih-l2stable-T} implies
%\eqref{eq-pih-l2stable}.
\end{proof}
Based on this lemma, we have the following corollary on the stability of $\Pi_{h}$ in the weighted $L^{2}$ norms.
\begin{corollary}
\label{cor:wl2stable}
Assume that $\Pi_h$ was defined as in \eqref{eqn:pih}, with the choice of $T_{x} \subset \Omega_{k}$, where $k$ is the minimal index of all the subdomains that intersect at $x$.
\begin{enumerate}
        \item $\Pi_h$ is stable in the standard $L^{2}$ norm:
\begin{equation}\label{eq:pih-l2stable}
        \|\Pi_h v\|_{L^{2}(\Omega)} \lesssim \|v\|_{L^{2}(\Omega)},\, \quad\forall v\in L^{2}(\Omega).
\end{equation}
        \item For any
piecewise-constant coefficient satisfying $\tau_1 \geq \tau_2 \geq \cdots \geq \tau_M$, $\Pi_{h}$ is stable in the $\tau$-weighted $L^{2}$ norm:
\begin{equation}\label{eq:pih-wl2stable}
        \|\Pi_h v\|_{0,\tau} \lesssim \|v\|_{0,\tau} \,, \quad\forall v\in L^{2}(\Omega).
\end{equation}
        \item In the worst scenario, for any piecewise-constant coefficient $\tau> 0$, we have
\begin{equation}\label{eq:pih-wl2stable-worsest}
        \|\Pi_h v\|_{0,\tau}^2 \lesssim \cJ(\tau)\|v\|_{0,\tau}^2 \,,\quad\forall v\in L^{2}(\Omega),
\end{equation}
where $\cJ(\tau)$ is the measure of the variation of $\tau$ defined by \eqref{eq:Jminmax}.
\end{enumerate}
\end{corollary}

\begin{remark}
Other projection operators that are stable in both the $L^{2}$ and weighted $L^{2}$ norms
are also available. For example, let
\begin{equation*}
\rho_x = \sum\limits_{T :\; x\in T}\rho_{T}
\end{equation*}
and define
\begin{equation*}
\left (\Pi_h v\right )(x) = \sum\limits_{T:\; x\in T}\frac{\rho_T}{\rho_x}\;
\left (Q_T v \right )(x).
\end{equation*}
Since $\rho_T \leq \rho_x$ implies $(\Pi_h v)^2(x) \leq  \sum\limits_{T:\; x\in T} (Q_T v)^2(x)$, it is
clear that $\Pi_h$ is stable in the $L^{2}$ norm.
The fact that it is also stable in the $\rho$-weighted $L^{2}$ norm follows from
$$
\rho_x (\Pi_h v)^2(x)
\leq \left( \sum\limits_{T:\; x\in T}\sqrt{\rho_T}\; \left (Q_T v \right )(x) \right)^2
\leq  \sum\limits_{T:\; x\in T} \rho_T (Q_T v)^2(x) \,.
$$
\end{remark}

%In the worst scenario, we have the following coefficient-depending estimate.
%\begin{corollary}
%\label{cor:wl2stable}
%Assume that $\Pi_h$ was defined as in \eqref{eq-pih-l2stable}. For any piecewise-constant coefficient $\tau> 0$, we have
%\begin{equation}\label{eq:pih-wl2stable-worsest}
%	\|\Pi_h v\|_{0,\tau} \lesssim \cJ(\tau)\|v\|_{0,\tau} \,,\forall v\in L^{2}(\Omega),
%\end{equation}
%where $$\cJ(\tau) = \frac{\max\{\tau_{i}\; :\; i = 1, \cdots, M \}}{ \min\{\tau_{j} \; : \; j =1, \cdots, M \}}.$$
%\end{corollary}
%The above result shows that $\Pi_h$ is stable simultaneously in the
%standard, and the $\rho$-weighted $L^{2}$ inner product.
%This is a unique feature of this interpolation operator, which allows
%us to obtain the approximation and $H^1$-stability estimates needed in
%the subsequent analysis.

Now we turn to study the approximation and stability properties in the $\omega$-weighted norms. It is standard (cf. \cite{Scott.R;Zhang.S1990}) that the $\Pi_{h}:H_{D}^{1}(\Omega) \to V_{h}$ has the following classical approximation and stability estimates:
\begin{equation*} %\label{eq-pih-approx}
\|v-\Pi_h v\|_{L^{2}(\Omega)} \lesssim h |v|_{H^1(\Omega)}, \,\qquad
|\Pi_h v|_{H^1(\Omega)} \lesssim |v|_{H^1(\Omega)} \,.
\end{equation*}
In the $\omega$-weighted norms, we have the following approximation and stability estimates.
\begin{lemma}
\label{lm:jh1stable}
        The interpolation $\Pi_{h}: H_{D}^{1}(\Omega) \to V_{h}$ satisfies the following approximation and stability estimates:
        \begin{equation} \label{eq-pih-japprox}
\|v-\Pi_h v\|^{2}_{0,\omega} \lesssim  \cJ(\omega) h |v|^{2}_{1,\omega} \,,
\end{equation}
\begin{equation} \label{eq-pih-h1-jstability}
|\Pi_h v|^{2}_{1,\omega} \lesssim \cJ(\omega) |v|^{2}_{1,\omega} \,.
\end{equation}
Moreover, if $\omega_1 \geq \omega_2 \geq \cdots \geq \omega_M$ and
$v \in \widetilde{H}_D^{1} (\Omega)$, we have the following estimates:
\begin{equation} \label{eq-pih-omega-approx}
\|v-\Pi_h v\|_{0,\omega} \lesssim h |\log h|^{\frac{1}{2}} |v|_{1,\omega} \,,
\end{equation}
\begin{equation} \label{eq-pih-h1-omega-stability}
|\Pi_h v|_{1,\omega} \lesssim |\log h|^{\frac{1}{2}} |v|_{1,\omega} \,.
\end{equation}
\end{lemma}

\begin{proof}
%Let $Q_h$ be the classical $L^{2}$ projection, see \cite{l2-proj}.
%The $L^{2}$-stability of $\Pi_h$ and the fact that $\Pi_h Q_h v = Q_h v$ imply
%$$
%\|v-\Pi_h v\|_{L^{2}(\Omega)} \leq
%\|v-Q_h v\|_{L^{2}(\Omega)} + \|\Pi_h(v-Q_h v)\|_{L^{2}(\Omega)} \lesssim \|v-Q_h v\|_{L^{2}(\Omega)} \,.
%$$
%Similarly,
%$$
%|\Pi_h v|_{H^1(\Omega)} \leq |Q_h v|_{H^1(\Omega)} + |\Pi_h(v - Q_h v)|_{H^1(\Omega)} \,.
%$$
%Using the inverse inequality in $V_h$ we conclude that
%$$
%h^{-1} \|v-\Pi_h v\|_{L^{2}(\Omega)} + |\Pi_h v|_{H^1(\Omega)}  \lesssim
%h^{-1} \|v-Q_h v\|_{L^{2}(\Omega)} + |Q_h v|_{H^1(\Omega)} \,.
%$$
%This implies the first two inequalities, since it is well known that $Q_h$ satisfies
%approximation and stability estimates analogous to \eqref{eq-pih-approx} and \eqref{eq-pih-h1-stability}.
Below, we give the proof of \eqref{eq-pih-omega-approx}-\eqref{eq-pih-h1-omega-stability}. The proof of the estimates \eqref{eq-pih-japprox}-\eqref{eq-pih-h1-jstability} is similar with minor changes.

Let $Q^\omega_h : H_{D}^{1}(\Omega) \to V_{h}$ be the $\omega$-weighted $L^{2}$-projection (see \cite{Bramble.J;Xu.J1991,Xu.J;Zhu.Y2008}). It satisfies
  $$\left\|v-Q_h^\omega v\right\|_{0,{\omega}}\lesssim h
  \left|\log h\right|^{\frac{1}{2}} \left\|v\right\|_{1,{\omega}}, \quad \forall v\in H_D^1(\Omega).$$
Then by triangle inequality, we have on each subdomain $\Omega_{m}$:
\begin{eqnarray*}
        \|v- \Pi_{h}v\|^2_{L^{2}(\Omega_{m})} &\lesssim& \|v - Q_h^\omega v\|^2_{L^{2}(\Omega_{m})} + \|\Pi_{h}(v - Q_h^\omega v)\|^2_{L^{2}(\Omega_{m})}\\
        &\lesssim& \sum_{k\le m} \|v - Q_h^\omega v\|^2_{L^{2}(\Omega_{k})},
\end{eqnarray*}
where in the last step we used Lemma~\ref{lm:l2stable} for the stability of $\Pi_{h}$ on the subdomain $\Omega_{m}$.
Therefore, we have
$$
        \|v - \Pi_{h}v\|_{0,\omega} \le \|v -  Q_h^\omega v\|_{0,\omega} \lesssim h
  \left|\log h\right|^{\frac{1}{2}} \left\|v\right\|_{1,{\omega}}.
$$
The inequality \eqref{eq-pih-omega-approx} then follows by the Poincar\'e-Friedrichs inequality on $\widetilde{H}_{D}^{1}(\Omega)$.

Similarly, to show the weighted $H^{1}$ stability \eqref{eq-pih-h1-omega-stability}, we have on each subdomain $\Omega_{m}$:
\begin{eqnarray*}
        |\Pi_{h}v|^2_{H^{1}(\Omega_{m})} &\lesssim& |Q_{h}^{\omega} v|^2_{H^{1}(\Omega_{m})} + | \Pi_{h}(v - Q_h^\omega v)|^2_{H^{1}(\Omega_{m})}\\
        &\lesssim& |Q_{h}^{\omega} v|^2_{H^{1}(\Omega_{m})} + h^{-2} \| \Pi_{h}(v - Q_h^\omega v)\|^2_{L^{2}(\Omega_{m})}\\
        &\lesssim&  |Q_{h}^{\omega} v|^2_{H^{1}(\Omega_{m})} + h^{-2} \sum_{k\le m}\|(v - Q_h^\omega v)\|^2_{L^{2}(\Omega_{k})}
\end{eqnarray*}
Then, the inequality \eqref{eq-pih-h1-omega-stability}
follows from the approximation and stability estimates  of $Q_{h}^{\omega}$ (see for example  \cite[Lemma 3.3]{Xu.J;Zhu.Y2008}). This completes the proof.
\end{proof}

%%%%%%%%%%%%%%%%%%%%%%%%%%%%%%%%%%%%%%%%%%%%%%%%%%%%%%%%%%%%%%%%%%%%%%%%%%%%%%%%
%  BPX
%%%%%%%%%%%%%%%%%%%%%%%%%%%%%%%%%%%%%%%%%%%%%%%%%%%%%%%%%%%%%%%%%%%%%%%%%%%%%%%%
\section{Multilevel Preconditioners}
\label{sec:multilevel}
In this section, we present the BPX and multigrid V-cycle preconditioners based on the subspace correction methods \cite{Xu.J1992a,Xu.J;Zikatanov.L2002}. We present the main results of the robustness of these preconditioners with respect to the jump in the coefficients.

Let $\cT_{0}$ be an initial conforming mesh which resolves the jump interfaces. We obtain a nested sequence of triangulation $\{\cT_{k}\}_{k=0}^{L}$  by a uniform refinement. Let $h_{k}$ be the mesh size of $\cT_{k}$ for $k =0 ,\cdots, L$, then we  have $h_{k} \simeq \gamma^{k} h_{0}$ for some $\gamma \in (0, 1)$, and $L \simeq |\log h_{L}|.$ For simplicity, we denote $h_{L} = h$. On each triangulation $\cT_{k}$, let $V_{k}$ be the corresponding finite element space over $\cT_{k}$. Then we obtain a sequence of nested spaces:
$$
        V_{0}\subset V_{1} \subset \cdots \subset V_{L} = V_{h}.
$$
These spaces defines a natural decomposition of $V_{h}$ as
$
        V_{h}= \displaystyle\sum_{k=0}^{L} V_{k}.
$
At each level $k =0, 1, \cdots, L$, we define the operator $A_k: V_k\to V_k$ by
$$(A_k w_k, v_k)=a(w_k,v_k),\;\; \forall w_k, v_k\in V_k$$
and simply denote $A =A_{L}$.  A key ingredient in analyzing the multilevel preconditioners is the stable decomposition derived below.
\subsection{Stable Decomposition}\label{sec-decomp}
With the help of the properties of the interpolation operator $\Pi_{h}$, we now show several stable results of the subspace decomposition described above. In the multilevel context, we will use the notation $\Pi_{k}:=\Pi_{h_{k}}$. Also, we notice that
$\Pi_{L} |_{V_{h}} = Id$, i.e., the restriction of $\Pi_{L}$ on the finite element space $V$ is identity.  In particular, for any $v\in V_{h}$, we consider the decomposition
\begin{equation}\label{eqn:decomp}
        v =  \sum_{k=0}^{L}  v_{k},
\end{equation}
where $v_{0} := \Pi_{0} v$ and $v_{k} : = (\Pi_{k} - \Pi_{k-1}) v \in V_{k}$ for $k=1, \cdots, L$.
Below, we discuss the stability of this decomposition in terms of the energy norm $\sqrt{a(\cdot, \cdot)} = \sqrt{(\cdot, \cdot)_{A}},$ which involves both the $\omega$-weighted $H^{1}$ semi-norm and the $\rho$-weighted $L^{2}$ norm. First, we consider the stability in terms of the $\omega$-weighted $H^{1}$ semi-norm.
%
%In summary, we consider the following different cases of the discontinuous coefficients $\omega$ and $\rho$:
%\begin{enumerate}
%	\item $\omega \simeq \omega_{0}$ where $\omega_{0}$ is a constant, but $\rho$ is discontinuous. In this case, we will use the standard approximation/stability estimates \eqref{eq-pih-approx}-\eqref{eq-pih-h1-stability}, and the stability in the $\rho$-weighted $L^{2}$ norm in Lemma~\ref{cor:wl2stable}.
%	\item Both $\omega$ and $\rho$ are discontinuous, but they have same distribution. Namely, if $\omega_{i} \ge \omega_{j}$ then $\rho_{i} \ge \rho_{j}$ and vice versa. This include the case when $\rho$ is a global constant. In this case, we have both Lemma~\ref{lm:wh1stable} in the $\omega$-weighted norms, and Corollary~\ref{cor:wl2stable} in the $\rho$-weighted $L^{2}$ norm.
%	\item The most general case, both $\omega$ and $\rho$ are discontinuous, and the coefficient distributions may be different.
%\end{enumerate}
% Then we have the following stable decomposition.
 \begin{lemma}\label{lm:stda}
 The decomposition \eqref{eqn:decomp} satisfies the following properties:
  \begin{enumerate}
    \item For any $v\in V_{h},$ there exist $v_k\in V_k \; (k=0, 1, \cdots,
    L)$ such that $v= \sum_{k=0}^L v_k$ and
    \begin{equation}\label{eq:stdav}
       |v_0|_{1,\omega}^2+\sum_{k=1}^L h_k^{-2}\|v_k\|_{0,\omega}^2
      \lesssim \mathcal{J}(\omega)  |v|_{1,\omega}^2,
    \end{equation}
    \item If $\omega_{1} \ge \cdots \ge \omega_{M}$, then for any $v\in \widetilde{V}_{h},$ there exist $v_k\in V_k \; (k=0, 1, \cdots,  L)$ such that $v=\sum_{i=0}^L  v_k$ and
    \begin{equation}\label{eq:stdatv}
        |v_0|_{1,\omega}^2+\sum_{k=1}^L h_k^{-2}\|v_k\|_{0,\omega}^2  \lesssim L^2 |v|_{1,\omega}^2
    \end{equation}
  \end{enumerate}
\end{lemma}
\begin{proof}
Given any $v\in V_{h},$ to show \eqref{eq:stdav}, we notice that
\begin{align*}
        |v_{0}|_{1,\omega}^{2} &+ \sum_{k=1}^L h_k^{-2}\|v_k\|_{0,\omega}^2\\
      &\lesssim \max_{k=1,\cdots, M} \{\omega_{k}\}  \left(|\Pi_{0 }v |_{H^{1}(\Omega)}^{2} + \sum_{k=1}^L h_k^{-2}\| (\Pi_{k} - \Pi_{k-1}) v\|_{L^{2}(\Omega)}^2\right).
\end{align*}
We have $| \Pi_{0} v |_{H^{1}(\Omega)} \le |v|_{H^{1}(\Omega)}$. To estimate the second term on the right hand side of above inequality, we use the fact that $\Pi_k$
is stable in $L^{2}$ and $\Pi_k Q_k = Q_k$, where $Q_{k}$ is the standard $L^{2}$-projection on $V_{k}$ for $k=0, 1, \cdots, L$.  Specifically, we have
$$
\begin{aligned}
\|&(\Pi_{k} - \Pi_{k-1}) v\|^2_{L^{2}(\Omega)}\\
& \lesssim
\|(Q_{k} - Q_{k-1}) v\|^2_{L^{2}(\Omega)} +
\|(\Pi_{k} - Q_{k}) v\|^2_{L^{2}(\Omega)} +
\|(\Pi_{k-1} - Q_{k-1}) v\|^2_{L^{2}(\Omega)} \\
& \lesssim
\|(Q_{k} - Q_{k-1}) v\|^2_{L^{2}(\Omega)} +
\|(I - Q_{k}) v\|^2_{L^{2}(\Omega)} +
\|(I - Q_{k-1}) v\|^2_{L^{2}(\Omega)} \\
& = 2 \|(I - Q_{k-1}) v\|^2_{L^{2}(\Omega)} \,.
\end{aligned}
$$
Therefore
$$
\begin{aligned}
\sum_{k=1}^L h_k^{-2} \|(\Pi_{k} - &\Pi_{k-1}) v\|^2_{L^{2}(\Omega)}
 \lesssim
\sum_{k=1}^L h_k^{-2} \|(I - Q_{k-1}) v\|^2_{L^{2}(\Omega)} \\
&\lesssim
h_1^{-2} \|(I - Q_{0}) v\|^2_{L^{2}(\Omega)} +
\sum_{k=2}^{L} (h_k^{-2} -h_{k-1}^{-2}) \|(I - Q_{k-1}) v\|^2_{L^{2}(\Omega)} \\
&=
\sum_{k=1}^L h_k^{-2} \|(I - Q_{k-1}) v\|^2_{L^{2}(\Omega)} -
\sum_{k=1}^L h_k^{-2} \|(I - Q_{k}) v\|^2_{L^{2}(\Omega)} \\
&=
\sum_{k=1}^L h_k^{-2} \|(Q_{k} - Q_{k-1}) v\|^2_{L^{2}(\Omega)}  \lesssim |v|^2_{H^{1}(\Omega)} \,.
\end{aligned}
$$
The estimate of the last sum is classical in the BPX theory, see \cite{Bramble.J;Pasciak.J;Xu.J1990,Oswald.P1999c}. Therefore, we have
$$
        |v_{0}|_{1,\omega}^{2} + \sum_{k=1}^L h_k^{-2}\|v_k\|_{0,\omega}^2 \lesssim (\max_{k} \omega_{k})\, |v|^2_{H^{1}(\Omega)} \le \cJ(\omega) |v|^2_{1,\omega}.
$$

To prove \eqref{eq:stdatv} for any $v\in \widetilde{V}_{h},$ by the approximation and stability estimates \eqref{eq-pih-omega-approx}-\eqref{eq-pih-h1-omega-stability}  of $\Pi_{k}$ ($k=0, 1, \cdots, L)$ in Lemma~\ref{lm:jh1stable} and triangle inequality, we obtain
\begin{eqnarray*}
    &&\left |\Pi_0  v\right |_{1,\omega}^2+\sum_{k=1}^L
    h_k^{-2}\|(\Pi_k -\Pi_{k-1})v\|_{0,\omega}^2 \lesssim
    \left(\sum_{k=0}^L |\log h_{k}|\right)  |v |_{1,\omega}^2\lesssim L^2 |v|_{1,\omega}^2.
 \end{eqnarray*}
This proves the inequality \eqref{eq:stdatv}.
\end{proof}
%\begin{remark}
%The estimate \eqref{eq:stdav} is not uniform for $d\geq 2$. For $d=2$, $L\simeq |\log h|$ and the growth of $c_2(L) =L^{2}\simeq |\log h|^{2}$ is acceptable. But for $d=3$, the constant $c_3(L)=2^L\simeq h^{-1}$ grows exponentially with respect to the number levels. For discontinuous coefficients problems, it seems unlikely to find a better decomposition with a better constants; see the counterexamples in~\cite{Bramble.J;Xu.J1991,Oswald.P1999c}.
%
%If the coefficients satisfy certain monotonicity, e.g. quasi-monotonicity (cf.~\cite{Dryja.M;Sarkis.M;Widlund.O1996,Petzoldt.M2002}) in the local patches, one can show that the interpolation operator defined above is stable in the energy norm without deterioration.
%\end{remark}

Now we establish the stable decomposition in the $\rho$-weighted $L^{2}$ norms. We have the following result.
\begin{lemma}\label{lm:wl2stable}
        The decomposition \eqref{eqn:decomp} satisfies the following properties:
        \begin{enumerate}
                \item For any $u\in V_{h}$, the decomposition $u_{0} = \Pi_{0}u$ and $u_{k} = (\Pi_{k} - \Pi_{k-1})u \in V_{k}$ for $k= 1, \cdots, L$ satisfies:
        \begin{equation}\label{eq:jl2stable}
                \|\Pi_0 u\|^2_{0,\rho} + \sum_{k=1}^L \|(\Pi_{k} - \Pi_{k-1}) u\|^2_{0,\rho}
                \lesssim \cJ(\rho) \|u \|^2_{0,\rho} \,.
        \end{equation}
                \item If the reaction coefficients satisfy $\rho_{1} \ge \rho_{2} \ge \cdots \ge \rho_{M}$, then for any $u\in V_{h}$, the decomposition $u_{0} = \Pi_{0}u$ and $u_{k} = (\Pi_{k} - \Pi_{k-1})u \in V_{k}$ for $k= 1, \cdots, L$ is stable:
        \begin{equation}\label{eq:l2stable}
\|\Pi_0 u\|^2_{0,\rho} + \sum_{k=1}^L \|(\Pi_{k} - \Pi_{k-1}) u\|^2_{0,\rho}
\lesssim \|u \|^2_{0,\rho} \,.
\end{equation}
\end{enumerate}
\end{lemma}
\begin{proof}
Below, we only give the detailed proof of \eqref{eq:l2stable}. The proof of \eqref{eq:jl2stable} can be reduced to show the estimate
$$
        \|\Pi_0 u\|^2_{L^{2}(\Omega)} + \sum_{k=1}^L \|(\Pi_{k} - \Pi_{k-1}) u\|^2_{L^{2}(\Omega)}\lesssim \|u \|^2_{L^{2}(\Omega)} \,,
$$
which is a special case of \eqref{eq:l2stable}.

Given any $u\in V_{h}$, by the stability \eqref{eq:pih-wl2stable} of $\Pi_{0}$ in the $\rho$-weighted $L^{2}$ norm, we have
$$
        \|\Pi_0 u\|^2_{0,\rho} \lesssim \|u\|^2_{0,\rho} \,.
$$
To estimate the summation term, we let $Q_k^\rho$ be the $\rho$-weighted $L^{2}$-projection on $V_k$.
Note that
$$
\|Q^\rho_{0} u\|^2_{0,\rho} + \sum_{k=1}^L \|(Q^\rho_{k} - Q^\rho_{k-1}) u\|^2_{0,\rho} = \|u\|^2_{0,\rho} \,.
$$
By the stability of \eqref{eq:pih-wl2stable} of $\Pi_{k}$ ($k=0,1,\cdots, L$) in the $\rho$-weighted $L^{2}$ norm, $\Pi_k Q_k^{\rho} = Q_k^{\rho}$ and triangle inequality, we obtain
$$
\begin{aligned}
\|(\Pi_{k} - \Pi_{k-1}) u\|^2_{0,\rho}
& \lesssim
\|(Q_{k}^{\rho} - Q_{k-1}^{\rho}) u\|^2_{0,\rho} +
\|(\Pi_{k} - Q_{k}^{\rho}) u\|^2_{0,\rho} +
\|(\Pi_{k-1} - Q_{k-1}^{\rho}) u\|^2_{0,\rho}.
\end{aligned}
$$
%where in the last step we used the orthogonal property $((I - Q_{k}^{\rho}) v, \; (Q_{k}^{\rho} - Q_{k-1}^{\rho}) v)_{0,\rho} =0.$
Therefore,
$$
\sum_{k=1}^L \|(\Pi_{k} - \Pi_{k-1}) u\|^2_{0,\rho} \lesssim
\|u\|^2_{0,\rho} + \sum_{k=0}^{L} \|(\Pi_k - Q^\rho_{k}) u\|^2_{0,\rho} \,.
$$

To bound the sum on the right, we need to introduce some additional
notation.  Let ${\widehat V}_k$ be the space of discontinuous
piecewise linear polynomials, associated with the same mesh as $V_k$,
and let ${\widehat Q}_k$ be the piecewise local $L^{2}$-projection onto ${\widehat V}_k$.
Note that by definition, $\Pi_k = \Pi_k {\widehat Q}_k$ and
$Q_k^{\rho} = Q_k^{\rho} {\widehat Q}_k$.
Set $v_k = (\Pi_k - Q^\rho_{k}) u$, $w_j=({\widehat Q}_j - {\widehat Q}_{j-1}) u$,
and fix $0<\sigma<\frac{1}{2}$.
We have
$$
\begin{aligned}
\sum_{k=0}^{L} \|(\Pi_k - Q^\rho_{k}) u\|^2_{0,\rho}
&= \sum_{k=0}^{L} \sum_{j \leq k} ((\Pi_k - Q^\rho_k) w_j, v_k)_{0,\rho} \\
%% &\lesssim \sum_{k=0}^{L-1} \sum_{j \leq k} \|(I - Q^\rho_k) w_j\|_{0,\rho} \|v_k\|_{0,\rho} \\
&\lesssim \sum_{k=0}^{L} \sum_{j \leq k} c_d^\sigma(h_j,h_k)\, h_k^\sigma \, |w_j|_{\sigma,\rho} \, \|v_k\|_{0,\rho}\\
&\lesssim \sum_{k=0}^{L} \sum_{j \leq k} \left(c_d(h_j,h_k)\, \frac{h_k}{h_j}\right)^\sigma \, \|w_j\|_{0,\rho} \, \|v_k\|_{0,\rho} \,.
\end{aligned}
$$
Here we used the (local) approximation property of $Q_{k}^{\rho}$, see \cite[Lemma 4.5]{Bramble.J;Xu.J1991}:
$$
\|(I - Q^\rho_k) w_j\| \lesssim c_d(h_j,h_k)\, h_k\, |w_j|_{1,\rho} \,,
$$
where $c_d(h_j,h_k) =\left\{\begin{array}{ll}
|\log \frac{h_j}{h_k}|^{\frac{1}{2}}, & d=2\\
\left(\frac{h_j}{h_k}\right)^{\frac{1}{2}}, & d=3.
\end{array}\right.$
Notice that
$$
c_d(h_j,h_k)\, \frac{h_k}{h_j} < \left(\sqrt{\gamma}\right)^{k-j} \qquad \text{with} \qquad \gamma < 1 \,.
$$
This implies
$$
\sum_{k=0}^L \|(\Pi_{k} - Q^\rho_{k}) u\|^2_{0,\rho} \lesssim \sum_{j=1}^L
\|({\widehat Q}_j - {\widehat Q}_{j-1}) u\|^2_{0,\rho}
\leq \|u\|^2_{0,\rho} \,.
$$
Thus, we obtain
\begin{equation*}%\label{eq:l2stable}
\|\Pi_0 u\|^2_{0,\rho} + \sum_{k=1}^L \|(\Pi_{k} - \Pi_{k-1}) u\|^2_{0,\rho}
\lesssim \|u \|^2_{0,\rho} \,.
\end{equation*}
This completes the proof.
\end{proof}

\subsection{BPX Preconditioning} \label{sec-bpx}
Now we are in position to discuss the performance of the multilevel preconditioners. For simplicity,  we introduce the mesh dependent coefficient on each level $k=1, \cdots, L$
\begin{equation} \label{omegak}
\omega_k = \omega + h_k^{2}\, \rho \,.
\end{equation}
%The embedding $V_k \subset V_L$ induces an interpolation matrix, which we
%denote by $\mathbb{P}_{k}$.
On each level $k = 1, \cdots, L$, let $R_{k}$ be a smoother based on the SPD operator $A_{k}$, and let $R_{0} = A_{0}^{-1}$.
We assume that the smoothers satisfy
$$
        (R_{k} u_{k}, u_{k}) \simeq h_k^{2} \|u_k\|^2_{0,\omega_k}
\qquad \forall u_k \in V_k\,, k=1,\cdots, L
$$
which holds for the Jacobi and Gauss-Seidel smoothers, as shown in Section~\ref{sec:pre}. Then the BPX preconditioner $B: V\to V$ is defined by
$$
        B = \sum_{k=0}^{L} R_{k} Q_{k},
$$
where $Q_{k}: V_{h}\to V_{k}$ be the $L^{2}$ projection on $V_{k}$ for $k=0,1,\cdots, L.$ The BPX preconditioner $B$ satisfies the following well-known identity (\cite{Widlund.O1992,Xu.J1992,Xu.J;Zikatanov.L2002}):
$$
        (B^{-1} v, v) = \inf_{\sum_{k=0}^{L} v_{k} = v}  \sum_{k=0}^{L} (R_{k}^{-1}v, v),\qquad \forall v\in V.
$$
Based on the assumption on $R_{k}$, it satisfies
\begin{equation}\label{eq:bpx}
        (B^{-1} v, v) \simeq \inf_{\sum_{k=0}^L v_k = v}
\left\{a(v_0, v_{0}) + \sum_{k=1}^L h_k^{-2} \|v_k\|^2_{0,\omega_{k}} \right\} \,.
\end{equation}
To analyze the BPX preconditioner, we make use of the following strengthened Cauchy Schwarz inequality.
\begin{lemma}[{Strengthened Cauchy Schwarz, cf. \cite[Lemma 6.1]{Xu.J1992a}}]
        \label{lm:scs}
        For $j, k=0, \cdots, L$ and $j\le  k$, we have
        \begin{equation}
        \label{eqn:scs}
                \int_{\Omega} \omega \nabla v_{k} \cdot \nabla v_{j} dx \lesssim \gamma^{k-j} (h_{k}^{-1}\|v_{k}\|_{0,\omega}) (h_{j}^{-1}\|v_{j}\|_{0,\omega}), \quad \forall v_{k}\in V_{k}, v_{j}\in V_{j}.
        \end{equation}
\end{lemma}

When $\rho=0$, the largest eigenvalue of the matrix $BA$ is
bounded by a constant, as shown for example in \cite{Xu.J;Zhu.Y2008}. For general
$\rho$ we only get a sub-optimal estimate.
\begin{lemma} \label{lemma-lambda-max}
The largest eigenvalue of $BA$ is independent of the coefficients
$\rho$ and $\omega$, but depends logarithmically on the mesh size:
$$
        (A u, u) \lesssim\, |\!\log h|\, (B^{-1} u, u), \quad \forall u \in V_{h},
$$
which implies $\lambda_{\max}(BA) \lesssim |\log h|.$
\end{lemma}
\begin{proof}
Given any $u\in V_{h}$, let $u = \sum_{k=0}^{L} u_{k}$ with $u_{k}\in V_{k}$ for ($k=0, \cdots, L$) be an arbitrary decomposition of $u$. Then by the Strengthened Cauchy Schwarz inequality \eqref{eqn:scs}, we obtain
\begin{align*}
        |u|_{1,\omega}^{2} &= \left|\sum_{k=0}^{L} u_{k} \right|_{1,\omega}^{2} \le 2\left(|u_{0}|_{1,\omega}^{2} + \sum_{k=1}^{L} \sum_{j=1}^{L} \int_{\Omega} \omega \nabla u_{k}\cdot\nabla u_{j} dx\right)\\
        &\lesssim |u_{0}|_{1,\omega}^{2} + \sum_{k=1}^{L} \sum_{j=1}^{L} \gamma^{|k-j|} \left(h_{k}^{-1} \|u_{k}\|_{0,\omega}\right) \left( h_{j}^{-1}\|u_{j}\|_{0,\omega}\right)\\
        &\lesssim |u_{0}|_{1,\omega}^{2} + \sum_{k=1}^{L} h_{k}^{-2} \|u_{k}\|_{0,\omega}^{2}.
\end{align*}
On the other hand, by the Schwarz inequality, we obtain
\begin{equation} \label{eq-l2-decomp}
        \|u\|^2_{L^{2}(\Omega_m)} \, \lesssim L \sum_{k=0}^L \|u_k\|^2_{L^{2}(\Omega_m)} \,,
\end{equation}
which implies that
$$
        \|u\|^2_{0,\rho} \, \lesssim L \sum_{k=0}^L \|u_k\|^2_{0,\rho}.
$$
%Combining these we get that any $u \in V_h$ satisfies
%$$
%\omega_m |u|^2_{H^1(\Omega_m)} + \rho_m \|u\|^2_{L^{2}(\Omega_m)}
%\lesssim
%L \inf_{u|_{\Omega_m} = \sum_{k=1}^L\! u_{k,\Omega_m}}
%\sum_{k=1}^L h_k^{-2}  \|u_{k,\Omega_m}\|^2_{0,\omega_k}.
%$$
Therefore, we have
\begin{align*}
        (Au, u)& = |u|_{1,\omega}^{2} + \|u\|_{0,\rho}^{2} \lesssim L\left(
        a(u_0,u_0) + \sum_{k=1}^L  h_k^{-2} \|u_k\|^2_{0,\omega_k}\right).
\end{align*}
Since the decomposition is arbitrary, we have $(A u, u) \lesssim\, |\!\log h|\, (B^{-1} u, u)$, which completes the proof.
\end{proof}
\begin{remark}
The estimate \eqref{eq-l2-decomp} can not be improved in general.
This can be seen by taking $u \in V_1$ and decomposing it using
$u_k = \frac{1}{L} u$ for $1 \leq k \leq L$.
\end{remark}

To estimate the smallest eigenvalue, we classify the coefficients in two different cases:
\begin{enumerate}
        \item[(C1)] The coefficients $\omega$ and $\rho$ have the same distribution. Namely,  if $\omega_{i} \ge \omega_{j}$ then $\rho_{i} \ge\rho_{j}$ and vice versa.
        \item[(C2)] The coefficients $\omega$ and $\rho$ have different distribution.
\end{enumerate}
In the case of (C1), we may label the subdomains based on the ordering $\omega_{1}\ge \omega_{2}\ge \cdots \ge \omega_{M}$. By the definition of $\Pi_{h}$, it satisfies simultaneously the stable decomposition \eqref{eq:l2stable} in the $\rho$-weighted $L^{2}$ norm, and the stable decomposition \eqref{eq:stdatv} in the $\omega$-weighted $H^{1}$ semi-norm. Based on these properties, we have the following result.
\begin{lemma} \label{lemma-bpx-rhoconst}
If the coefficients $\omega$ and $\rho$ satisfy (C1), then
$$
 (B^{-1}u, u) \,\lesssim\, |\!\log h|^2\, (Au\, u), \qquad \forall u \in \widetilde{V}_{h},
$$
which implies $\lambda_{m_{0}+1}(BA) \gtrsim |\!\log h|^2.$
\end{lemma}
\begin{proof}
For any $u\in \widetilde{V}_{h}$, we consider the decomposition
$$
u = \Pi_0 u + \sum_{k=1}^L (\Pi_{k} - \Pi_{k-1}) u \,,
$$
i.e. $u_0 = \Pi_0 u$, $u_k = (\Pi_{k} - \Pi_{k-1}) u$ for $1 \le k \le L$.
As a direct consequence of the stable decomposition \eqref{eq:stdatv} and \eqref{eq:l2stable}, we obtain
$$
        a(u_{0}, u_{0}) + \sum_{k=1}^L h_k^{-2} \|u_{k}\|^2_{0,\omega_{k}}
\lesssim L^2\, a(u, u) \,.
$$
By \eqref{eq:bpx}, this implies $(B^{-1} u, u) \lesssim |\log h|^{2} (A u, u)$. The estimate of $\lambda_{m_{0} + 1}$ then follows by noticing that $\dim(\widetilde{V}_{h}) = \dim(V_{h}) - m_{0}$ and the min-max  principle (cf. Remark~\ref{rk:minmax}).
\end{proof}
%Therefore, if the coefficients $\omega$ and $\rho$ satisfy (C1), we have the following estimate on the effective condition number of the BPX preconditioner.

Now we turn to discuss the case (C2) when $\omega$ and $\rho$ have different distribution, e.g., there exists at least a pair of (neighboring) subdomains $\Omega_{i}$ and $\Omega_{j}$ on which $\omega_{i} >\omega_{j}$ but $\rho_{i} < \rho_{j}$. In this case, the interpolation operator $\Pi_{h}$ does not satisfy the simultaneous stability in both the $\rho$-weighted $L^{2}$ norm and $\omega$-weighted $H^{1}$ semi norm. Therefore, in this case, we only get some pessimistic estimates which depend on the jumps in the coefficients.

When $\cJ(\omega) < \cJ(\rho)$, we should label the subdomains based on the order of $\rho$ to guarantee the $\rho$-weighted $L^{2}$ stability \eqref{eq:l2stable}. Note that this includes the case when $\rho =0$ in some subdomains, but not globally 0. While for the $\omega$-weighted approximation and stability estimate, we apply the decomposition \eqref{eq:stdav} instead of \eqref{eq:stdatv}.
\begin{lemma} \label{lemma-bpx-jomega}
        If the coefficients $\omega$ and $\rho$ satisfy (C2)  and $\cJ(\omega) \le \cJ(\rho)$, then
        $$
                (B^{-1} u, u) \,\lesssim\, \cJ(\omega) (Au, u), \qquad \forall u\in V_{h},
        $$
        which implies that $\lambda_{\min}(BA) \gtrsim \cJ^{-1}(\omega).$ In particular, if $\omega$ is a global constant, then $\lambda_{\min}(BA) \gtrsim 1$, which is independent of the coefficient $\rho.$
\end{lemma}
\begin{proof}
        Given any $u\in V_{h}$, we define the decomposition $u = \sum_{k=0}^{L} u_{k}$ as $u_{0} = \Pi_{0} u$, and $u_{k} = (\Pi_{k} -\Pi_{k-1}) u$ for $k=1, \cdots, L$. Since the coefficient $\rho$ satisfies $\rho_{1} \ge \cdots \ge \rho_{M}$, this decomposition satisfies \eqref{eq:l2stable}. On the other hand, since $\omega$ and $\rho$ have different distribution, we can not apply \eqref{eq:stdatv} in this case, but we still have the stable decomposition \eqref{eq:stdav}. The conclusion then follows by \eqref{eq:bpx}, \eqref{eq:l2stable} and \eqref{eq:stdav}.
\end{proof}

On the other hand, if $\cJ(\omega) > \cJ(\rho)$, then we should label the subdomains based on the ordering of $\omega$ to guarantee the stable decomposition \eqref{eq:stdatv} in the $\omega$-weighted $H^{1}$ semi-norm. For the stable decomposition in term of $\rho$-weighted $L^{2}$ norm, we can not apply \eqref{eq:l2stable} directly, but we may use the estimate \eqref{eq:jl2stable}.
So in this case, we have the following result.
\begin{lemma} \label{lemma-bpx-jrho}
If the coefficients $\omega$ and $\rho$ satisfies (C2) and $\cJ(\omega)>\cJ(\rho)$, then
$$
 (B^{-1}u, u)\,\lesssim\, \max\{\cJ(\rho), |\!\log h|^2\}\, (Au, u), \qquad \forall u\in \widetilde{V}_h,
$$
which implies $\lambda_{m_{0}+1} (BA) \gtrsim \min\{ \cJ^{-1}(\rho), |\log(h)|^{-2}\}.$
\end{lemma}

In summary, we have the following results for the BPX preconditioner.
\begin{theorem}\label{thm:bpx}
The BPX preconditioner $B$ satisfies:
\begin{enumerate}
        \item If the coefficients $\omega$ and $\rho$ satisfy (C1), the $m_{0}$-th effective condition number of $BA$ is independent of the jumps in $\omega$ and $\rho$:
$$
\kappa_{m_0} (BA) \lesssim |\!\log h|^3 \,.
$$
Here $m_0 = |I|$, is the number of floating subdomains. In this case, it recovers essentially the main result in \cite[Lemma 4.2]{Xu.J;Zhu.Y2008}. The only difference is here the $m_{0}$-th effective condition number of $BA$ has an additional $|\log h|$ factor from Lemma~\ref{lemma-lambda-max}, which is a result of $\rho \neq 0$.
        \item If the coefficients $\omega$ and $\rho$ satisfy (C2), with $\cJ(\omega) > \cJ(\rho)$, the  $m_{0}$-th effective condition number of $BA$ is independent of the jump in $\omega$:
$$
        \kappa_{m_0} (BA) \lesssim \max\{\cJ(\rho) |\log h|, \; |\!\log h|^3\} \,.
$$
        \item If the coefficients $\omega$ and $\rho$ satisfy (C2), with $\cJ(\omega) \le \cJ(\rho)$, the condition number of $BA$ is independent of the jump in $\rho$:
        $$
                \kappa(BA) \lesssim \cJ(\omega) |\log h|.
        $$
        In particular, if $\omega$ is a global constant, then the condition number of $BA$ is independent of the jumps in both of $\omega$ and $\rho$.
\end{enumerate}
\end{theorem}

\subsection{Multigrid V-cycle} \label{sec-mg}
Now we consider the Multigrid V-cycle as a solution
 algorithm and as a preconditioner to our original elliptic problem
\eqref{eq:model}. We first introduce some standard notation. For each level $k=0, 1,\dots, L,$ we define the projections $P_k:  V_{h}\to V_k$ by
$$a(P_k u, v_k)=a(u,v_k)\;\; \forall v_k\in
  V_k.$$ At each level, let $R_{k}: V_{k}\to V_{k}$ be the smoothing operator. Here we use point Gauss-Seidel as the smoother. Then that standard multigrid V-cycle algorithm solves \eqref{eq:eq}  by the
iterative method
$$u_k\leftarrow u_k +B_k(f_k-A_k u_k),$$ where the operator $B_k: V_k\to
V_k$ is defined recursively as follows:
\begin{alg}[V-cycle]
\label{alg:mg}
Let $B_0=A_0^{-1},$ for $k>0$ and
$g\in  V_k,$ define
\begin{enumerate}
  \item Presmoothing : $w_1=R_k g;$
  \item Correction: $w_2=w_1+ B_{k-1}Q_{k-1}(g-A_k w_1);$
  \item Postsmoothing: $B_{k} g=w_2+R_k^*(g-A_k w_2).$
\end{enumerate}
\end{alg}
We denote $B_{L} = B$ for simplicity. Following the same analysis in \cite{Xu.J;Zhu.Y2008}, it is clear that $\lambda_{\max}(BA) \le 1$. To estimate the smallest eigenvalue of $BA$, we consider the error propagation operator $I-BA$. By the XZ-identity (cf. \cite{Xu.J;Zikatanov.L2002}),
we can get the following estimate, which is a
straightforward generalization of  \cite[Lemma 5.2]{Xu.J;Zhu.Y2008}.
\begin{lemma}\label{lm:mg}
        For any $v\in V_{h}$, consider the decomposition $v = \Pi_{0} v + \sum_{l=0}^{L}(\Pi_{k} - \Pi_{k-1}) v$, then the error propagate operator $I-BA$ satisfies the following estimate
        $$
                \|I-BA\|_{A} = \frac{c_{0}}{1+c_{0}} = 1- \frac{1}{1+ c_{0}},
        $$
        where
        \begin{equation}
        \label{eq:c0}
                c_{0} \lesssim \sup_{\substack{v\in V_{h}\\
                \|v\|_{A} =1}}  \left( \sum_{k=0}^{L} \|P_{k}v - \Pi_{k} v\|_{A}^{2} + \sum_{k=1}^{L} h_{k}^{-2}  \|(\Pi_{k} - \Pi_{k-1}) v\|_{0,\omega_{k}}\right).
        \end{equation}
        If we restrict to the subspace $\widetilde{V}_{h}$, we have a similar estimate:
        $$
                \|(I-BA)|_{\widetilde{V}_{h}}\|_{A} = \frac{\tilde{c}_{0}}{1+\tilde{c}_{0}} = 1- \frac{1}{1+ \tilde{c}_{0}},
        $$
        where
        \begin{equation}
        \label{eq:tc0}
                \tilde{c}_{0} \lesssim \sup_{\substack{v\in \widetilde{V}_{h}\\
                \|v\|_{A} =1}}  \left( \sum_{k=0}^{L} \|P_{k}v - \Pi_{k} v\|_{A}^{2} + \sum_{k=1}^{L} h_{k}^{-2}  \|(\Pi_{k} - \Pi_{k-1}) v\|_{0,\omega_{k}}\right).
        \end{equation}
\end{lemma}
From Lemma~\ref{lm:mg}, as in \cite{Xu.J;Zhu.Y2008} we can deduce by min-max principle (cf. Remark~\ref{rk:minmax}):
\begin{align*}
        \lambda_{\min}(BA) &= \min_{\substack{v\in V_{h}\\ v\neq 0}} \frac{a(BA v, v)}{a(v,v)} \ge \frac{1}{1+c_{0}}, \\
        \lambda_{m_{0} + 1}(BA)  &\ge \min_{\substack{v\in \widetilde{V}_{h}\\ v\neq 0}} \frac{a(BA v, v)}{a(v,v)} \ge \frac{1}{1+ \tilde{c}_{0}},
\end{align*}
where $m_{0} = |I|$ is the number of floating subdomains.
%\begin{lemma}
%%Introduce the notation $\|u\|_a^2 = |u|_{1,\omega}^2 + \|u\|_{0,\rho}^2$, and
%%the $a(\cdot,\cdot)$-orthogonal projectors $P_k: V_h \mapsto V_k$, defined by
%%$$
%%a(P_k u_h, v_k) = a(u_h, v_k) \qquad \text{for all } v_k \in V_k\,.
%%$$
%Suppose that the interpolation operators $\{\Pi_{k}\}$ satisfy the estimate
%\begin{equation} \label{eq-mgstab}
%\sum_{k=0}^L \|(P_{k} - \Pi_{k}) u\|^2_a
%+ \sum_{k=1}^L h_k^{-2} \|(\Pi_{k} - \Pi_{k-1}) u\|^2_{0,\omega}
%+ \|(\Pi_{k} - \Pi_{k-1}) u\|^2_{0,\rho}
%\lesssim C(h)\, \|u\|_a^2
%\end{equation}
%with some coefficient $C(h)$ possibly depending on $h$, but independent of $u \in V_h$.
%
%Then, the asymptotic convergence factor of Multigrid V-cycle with a Gauss-Seidel
%smoother satisfies
%$$
%\|I-BA\|_A \lesssim 1 - \frac{1}{1+C(h)} \,,
%$$
%while the condition number of the Multigrid-preconditioned problem is
%$$
%\kappa (\mathbb{B}\,\mathbb{A}) \lesssim C(h) \,.
%$$
%If \eqref{eq-mgstab} holds for $u \in \tilde{V}_h$ with a better coefficient $\tilde{C}(h)$,
%then we get an improved estimate for the effective condition number:
%$$
%\kappa_{m_0} (\mathbb{B} \mathbb{A}) \lesssim \tilde{C}(h) \,.
%$$
%\end{lemma}
According to the above result, the convergence of the multigrid V-cycle method, and the condition number estimate of the multigrid preconditioner rely on the estimate on the constant $c_{0}$; while the estimate on the effective condition number relies on the estimate on $\tilde{c}_{0}$. Both of these estimates follow from the stable decompositions \eqref{eq:c0} and \eqref{eq:tc0}.  Now, based on the discussion for the BPX preconditioner case, we can obtain similar results for the multigrid V-cycle.

\begin{theorem} \label{th-mg-omega1}
The multigrid preconditioner $B$ defined in Algorithm~\ref{alg:mg} satisfies:
\begin{enumerate}
        \item If the coefficients $\omega$ and $\rho$ satisfy (C1), the $m_{0}$-th effective condition number of $BA$ is independent of the jumps in $\omega$ and $\rho$:
$$
\kappa_{m_0} (BA) \lesssim |\!\log h|^2 \,.
$$
Here $m_0 = |I|$, is the number of floating subdomains.
        \item If the coefficients $\omega$ and $\rho$ satisfy (C2), with $\cJ(\omega) > \cJ(\rho)$, the  $m_{0}$-th effective condition number of $BA$ is independent of the jump in $\omega$:
$$
        \kappa_{m_0} (BA) \lesssim \max\{\cJ(\rho) |\log h|, \; |\!\log h|^2\} \,.
$$
        \item If the coefficients $\omega$ and $\rho$ satisfy (C2), with $\cJ(\omega) \le \cJ(\rho)$, the condition number of $BA$ is independent of the jump in $\rho$:
        $$
                \kappa(BA) \lesssim \cJ(\omega)|\log h|.
        $$
        In particular, if $\omega$ is a global constant, then the condition number of $BA$ is independent of the jumps in both of $\omega$ and $\rho$.
\end{enumerate}
\end{theorem}
\section{Numerical experiments} \label{sec-numerical}
This section contains a set of numerical experiments performed with a version
of the finite element library MFEM \cite{mfem}, which illustrate
the convergence theory developed in the preceding sections.
We focus on the commonly used $V(1,1)$-cycle Multigrid method,
and use a symmetric Gauss-Seidel iteration as a smoother.
The same smoother was also used in the BPX algorithm, whose optimal
implementation can be found in \cite{Bramble.J;Pasciak.J;Xu.J1990}.

The jump-independent estimates of the effective condition number in
Section \ref{sec-bpx} and Section \ref{sec-mg} imply that a
preconditioned conjugate gradient (PCG) acceleration will result in a
solver which is optimal with respect top the mesh size.
To investigate this, we report the number of PCG iterations needed to
reduce the relative residual by a factor of $10^{-12}$.
We use the abbreviations GS-CG, BPX-CG and MG-CG to denote the
symmetric Gauss-Seidel, BPX and Multigrid preconditioners respectively.

\begin{figure}
\centerline{\includegraphics[height=2in]{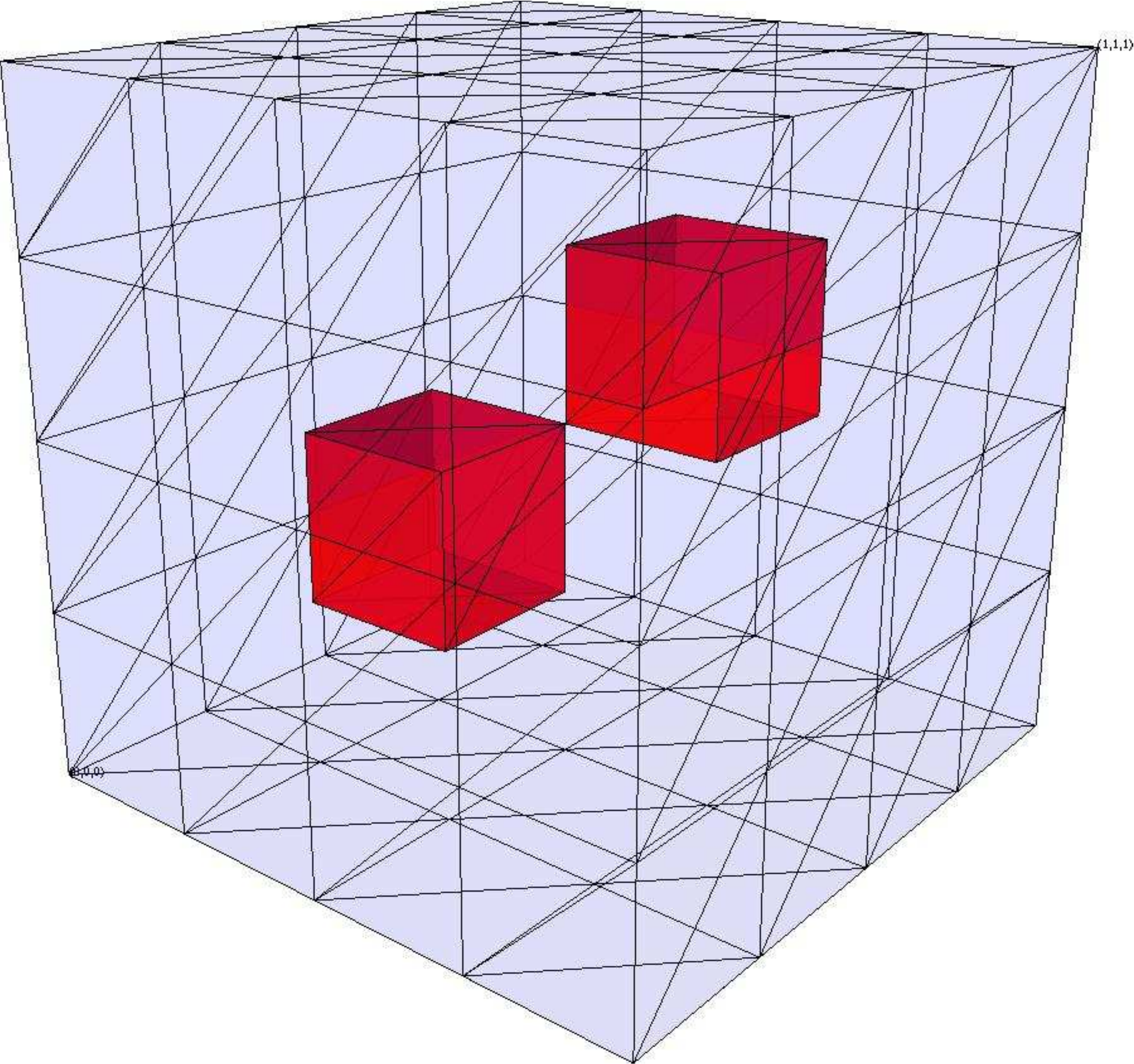}}
\caption{Geometry of the two material subdomains test problem.}
\label{fig-touch-cubes}
\end{figure}

We run a simple test problem on the unit cube, which is a model of a
soft/hard material enclosure.
As in \cite{Xu.J;Zhu.Y2008}, we only consider the two material subdomains case
pictured in Figure \ref{fig-touch-cubes}, and we let $\Omega_2$ be the
union of the two internal cubes, while $\Omega_1$ denotes the rest of
the domain.
The problem was discretized with linear finite elements on regular
tetrahedral mesh, using zero Dirichlet boundary conditions on the boundary.
The right-hand side in all the tests, was chosen to correspond to the
unit constant function, and the initial guess was a vector of zeros.
Some of the computed numerical solutions are plotted in Figure \ref{fig-sol}.

\begin{figure}
\centerline{
\includegraphics[height=2in]{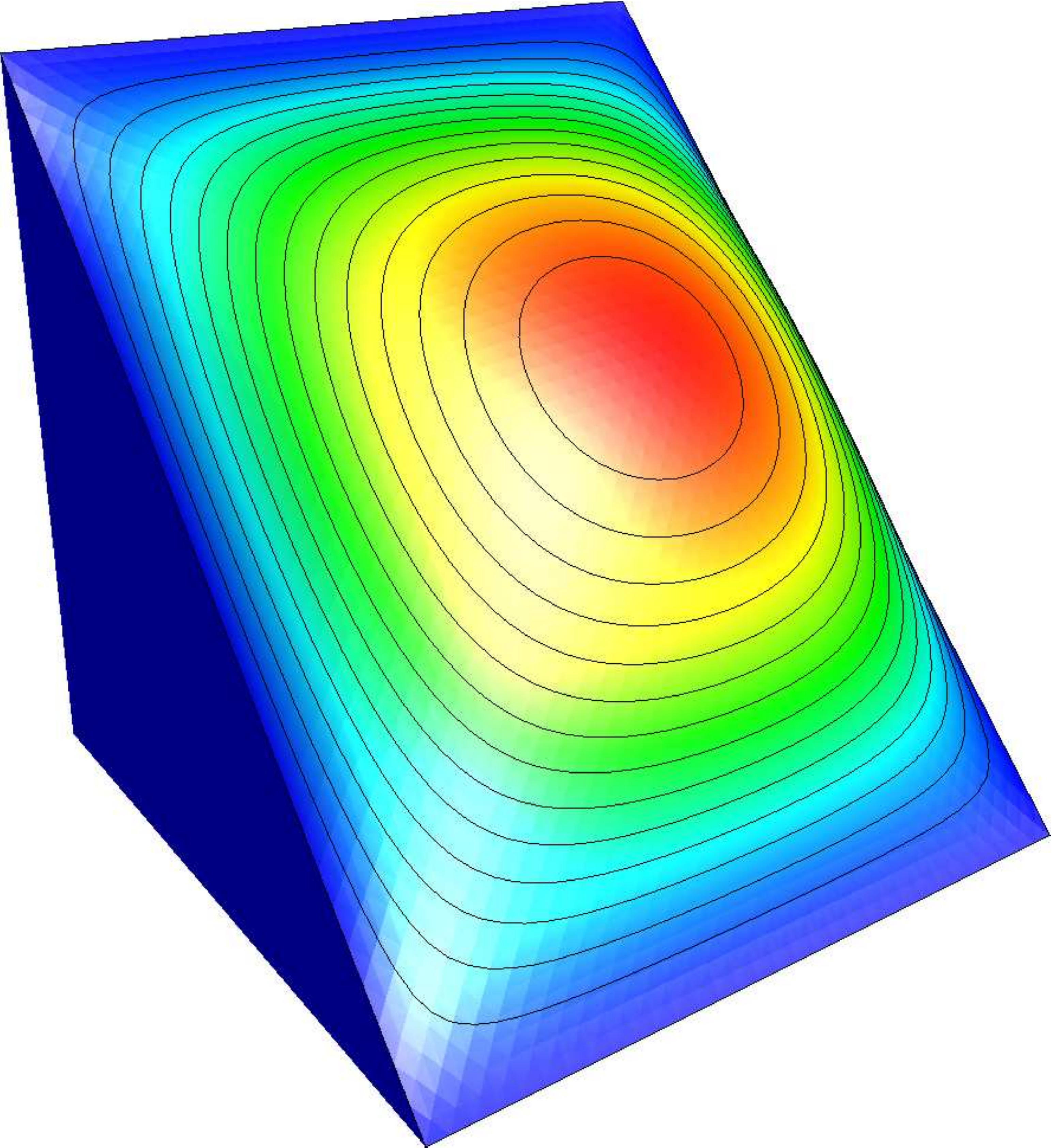}
\includegraphics[height=2in]{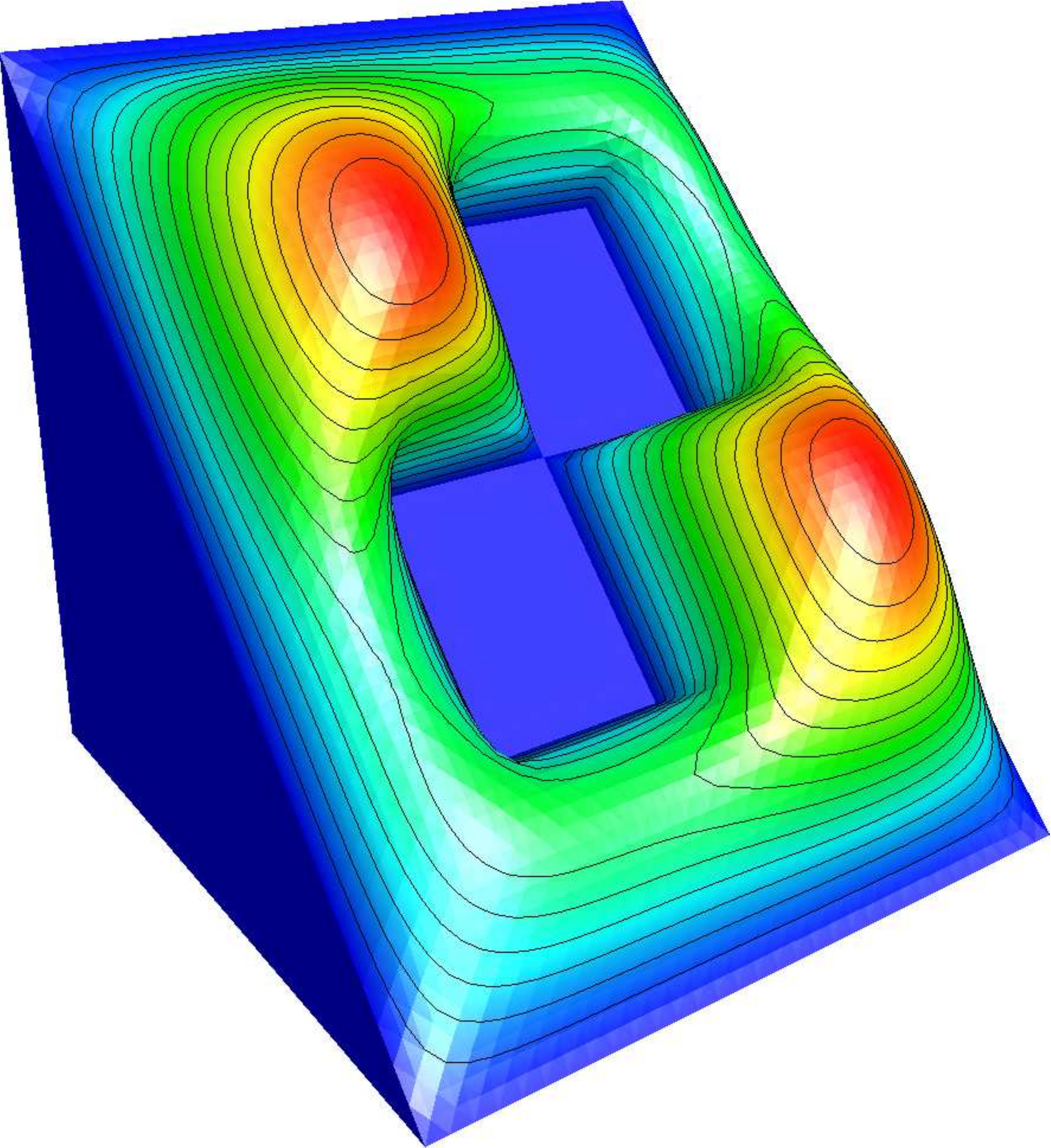}
\includegraphics[height=2in]{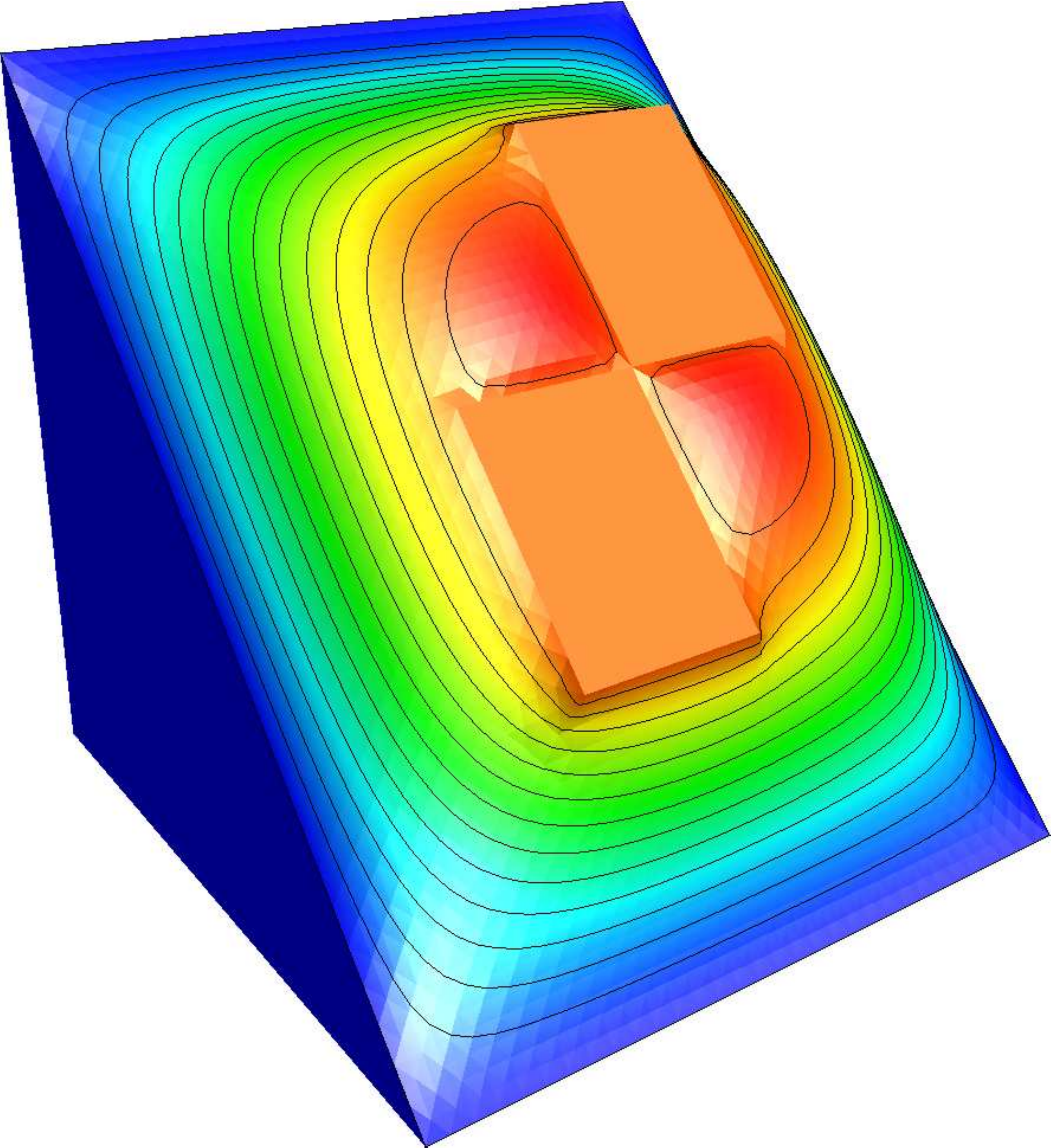}
}
\caption{Approximate solutions in a cut inside the domain corresponding to
$\omega=1$, $\rho=1$ (left); $\omega=1$, $\rho_1=1$, $\rho_2=10^8$ (center);
and $\omega_1=1$, $\omega_2=10^8$, $\rho=1$ (right).}
\label{fig-sol}
\end{figure}

Since we can always rescale the original equation, we can assume,
without a loss of generality, that $\omega_2=1$.
In particular, when $\omega$ is a constant we will set it equal to one.
This is the case that we set to explore first.

%%%%%%%%%%%%%%%%%%%%%%%%%%%%%%%%%%%%%%%%%%%%%%%%%%%%%%%%%%%%%%%%%%%%%%%%%%%%%%%%
%  Numerical experiments - Constant omega
%%%%%%%%%%%%%%%%%%%%%%%%%%%%%%%%%%%%%%%%%%%%%%%%%%%%%%%%%%%%%%%%%%%%%%%%%%%%%%%%
\subsection{The case of constant $\omega$}
To restrict the parameter range, we first set $\omega=1$ and allow
$\rho_1$ and $\rho_2$ to vary independently in $\{0\} \cup [10^{-8},10^8]$.
The results of Gauss-Seidel preconditioned conjugate gradient are
presented in Table \ref{table-omega1-gs-l34}.
Here $\ell$ denotes the refinement level corresponding to problem size
$N$ and mesh size $h$.
We also use the ``scientific'' notation {\tt 1e+p} to denote the number $10^p$.

{\Small
\begin{table}[!ht]
\begin{center}
{\def\arraystretch{1.2}
\begin{tabular}{|r|c||c|c|c|c|c|c|c|c|c|c|}
\hline
\multicolumn{2}{|c||}{$\omega_1=1$} & \multicolumn{10}{c|}{$\rho_1$} \\
\cline{3-12}
\multicolumn{2}{|c||}{$\omega_2=1$} & {\tt 0} & {\tt 1e-8} & {\tt 1e-6} & {\tt 1e-4} & {\tt 1e-2} & {\tt 1e-0} & {\tt 1e+2} & {\tt 1e+4} & {\tt 1e+6} & {\tt 1e+8} \\
\hline
\hline
\multicolumn{12}{|c|}{$\ell=3$, $h^2 \approx$ {\tt 1e-3}, $N=35,937$}  \\
\hline
& {\tt 0}    & 62 & 62 & 62 & 62 & 62 & {\bf 66} & 50 & 20 & 19 & 19 \\
& {\tt 1e-8} & 62 & 62 & 62 & 62 & 62 & {\bf 66} & 50 & 20 & 19 & 19 \\
& {\tt 1e-6} & 62 & 62 & 62 & 62 & 62 & {\bf 66} & 50 & 20 & 19 & 19 \\
& {\tt 1e-4} & 62 & 62 & 62 & 62 & 62 & {\bf 66} & 50 & 20 & 19 & 19 \\
& {\tt 1e-2} & 62 & 62 & 62 & 62 & 62 & {\bf 66} & 50 & 20 & 19 & 19 \\
$\rho_2$
& {\tt 1e-0} & 66 & 66 & 66 & 66 & 66 & {\bf 62} & 50 & 20 & 19 & 19 \\
& {\tt 1e+2} & 69 & 69 & 69 & 69 & 69 & {\bf 68} & 46 & 19 & 19 & 18 \\
& {\tt 1e+4} & 63 & 63 & 63 & 63 & 63 & {\bf 62} & 45 & 10 & 15 & 14 \\
& {\tt 1e+6} & 60 & 60 & 60 & 60 & 60 & {\bf 60} & 45 & 13 & 15 & 15 \\
& {\tt 1e+8} & 60 & 60 & 60 & 60 & 60 & {\bf 60} & 45 & 13 & 15 & 15 \\
\hline
\hline
\multicolumn{12}{|c|}{$\ell=4$, $h^2 \approx$ {\tt 2.5e-4}, $N=274,625$}  \\
\hline
& {\tt 0}    & 120 & 120 & 120 & 120 & 120 & {\bf 120} & 96 & 38 & 36 & 36 \\
& {\tt 1e-8} & 120 & 120 & 120 & 120 & 120 & {\bf 120} & 96 & 38 & 36 & 36 \\
& {\tt 1e-6} & 120 & 120 & 120 & 120 & 120 & {\bf 120} & 96 & 38 & 36 & 36 \\
& {\tt 1e-4} & 120 & 120 & 120 & 120 & 120 & {\bf 120} & 96 & 38 & 36 & 36 \\
& {\tt 1e-2} & 120 & 120 & 120 & 120 & 120 & {\bf 120} & 96 & 38 & 36 & 36 \\
$\rho_2$
& {\tt 1e-0} & 121 & 121 & 121 & 121 & 121 & {\bf 120} & 96 & 38 & 36 & 36 \\
& {\tt 1e+2} & 133 & 133 & 133 & 133 & 133 & {\bf 132} & 85 & 37 & 35 & 35 \\
& {\tt 1e+4} & 123 & 123 & 123 & 123 & 123 & {\bf 122} & 89 & 13 & 14 & 15 \\
& {\tt 1e+6} & 117 & 117 & 117 & 117 & 117 & {\bf 116} & 89 & 14 & 14 & 15 \\
& {\tt 1e+8} & 117 & 117 & 117 & 117 & 117 & {\bf 117} & 89 & 14 & 15 & 15 \\
\hline
\end{tabular}}
\medskip
\caption{Number of GS-CG iterations when $\omega=1$.}
\label{table-omega1-gs-l34}
\end{center}
\end{table}
}

Several things are apparent from Table \ref{table-omega1-gs-l34}.
First, when $\rho h^2 \gtrsim \omega$ (the lower right corner in the
tables) the problem is well conditioned and GS-CG is an efficient
solver.
Second, the convergence is largely independent of the jumps in $\rho$
and the number of iterations is proportional to $h^{-1}$, as expected
by Theorem \ref{th-gs}.
Finally, it is clear that the problem of hard enclosure, when $\rho_2 > \rho_1$, is
more difficult than the one of soft enclosure ($\rho_1 > \rho_2$).

Motivated by the above observations, we choose to restrict our further
experiments to the case $\omega=1$, $\rho_1=1$.
This way the results have a more compact form, as can be seen by comparing
Table \ref{table-omega1-gs-l34} and Table \ref{table-omega1-gs}.

\begin{table}[!ht]
\begin{center}
{\def\arraystretch{1.2}
\begin{tabular}{|r|r|c|c|c|c|c|c|c|c|c|c|}
\hline
  &  & \multicolumn{10}{c|}{$\rho_2$} \\
\cline{3-12}
$\ell$ &  $N$ & {\tt 0} & {\tt 1e-8} & {\tt 1e-6} & {\tt 1e-4} & {\tt 1e-2} & {\tt 1e-0} & {\tt 1e+2} & {\tt 1e+4} & {\tt 1e+6} & {\tt 1e+8} \\
\hline
1 &       729 & 18 & 18 & 18 & 18 & 18 & 18 & 18 & 16 & 16 & 16 \\
2 &     4,913 & 36 & 36 & 36 & 36 & 36 & 36 & 38 & 34 & 34 & 34 \\
3 &    35,937 & 66 & 66 & 66 & 66 & 66 & 62 & 68 & 62 & 60 & 60 \\
4 &   274,625 & 120 & 120 & 120 & 120 & 120 & 120 & 132 & 122 & 116 & 117 \\
\hline
\end{tabular}}
\medskip
\caption{Number of GS-CG iterations when $\omega=1$ and $\rho_1=1$.}
\label{table-omega1-gs}
\end{center}
\end{table}

In Tables \ref{table-omega1-bpx-cg}--\ref{table-omega1-mg-cg} we
demonstrate the performance of the BPX preconditioner and the
Multigrid solver and preconditioner on problems with constant
$\omega$.
The results indicate that BPX-CG may have a nearly-optimal convergence
rate, see Theorem \ref{thm:bpx}, while the convergence of Multigrid is
optimal.

\begin{table}[!ht]
\begin{center}
{\def\arraystretch{1.2}
\begin{tabular}{|r|r|c|c|c|c|c|c|c|c|c|c|}
\hline
  &  & \multicolumn{10}{c|}{$\rho_2$} \\
\cline{3-12}
$\ell$ &  $N$ & {\tt 0} & {\tt 1e-8} & {\tt 1e-6} & {\tt 1e-4} & {\tt 1e-2} & {\tt 1e-0} & {\tt 1e+2} & {\tt 1e+4} & {\tt 1e+6} & {\tt 1e+8} \\
\hline
1 &       729 & 20 & 20 & 20 & 20 & 20 & 20 & 19 & 19 & 19 & 18 \\
2 &     4,913 & 27 & 27 & 27 & 27 & 27 & 27 & 27 & 30 & 31 & 30 \\
3 &    35,937 & 31 & 31 & 31 & 31 & 31 & 31 & 31 & 35 & 37 & 37 \\
4 &   274,625 & 33 & 33 & 33 & 33 & 33 & 33 & 33 & 38 & 43 & 42 \\
5 & 2,146,689 & 35 & 35 & 35 & 35 & 35 & 35 & 35 & 39 & 47 & 47   \\
\hline
\end{tabular}}
\medskip
\caption{Number of BPX-CG iterations when $\omega=1$ and $\rho_1=1$.}
\label{table-omega1-bpx-cg}
\end{center}
\end{table}

{\SMALL
\begin{table}[!ht]
\begin{center}
{\def\arraystretch{1.2}
\begin{tabular}{|r|c|c|c|c|c|c|c|c|c|c|}
\hline
  &  \multicolumn{10}{c|}{$\rho_2$} \\
\cline{2-11}
$\ell$ & {\tt 0} & {\tt 1e-8} & {\tt 1e-6} & {\tt 1e-4} & {\tt 1e-2} & {\tt 1e-0} & {\tt 1e+2} & {\tt 1e+4} & {\tt 1e+6} & {\tt 1e+8} \\
\hline
1 & 16 (0.17) & 16 (0.17) & 16 (0.17) & 16 (0.17) & 16 (0.17) & 16 (0.17) & 16 (0.16) & 17 (0.16) & 17 (0.17) & 17 (0.17) \\
2 & 18 (0.20) & 18 (0.20) & 18 (0.20) & 18 (0.20) & 18 (0.20) & 18 (0.20) & 18 (0.20) & 22 (0.27) & 23 (0.28) & 23 (0.28) \\
3 & 18 (0.21) & 18 (0.21) & 18 (0.21) & 18 (0.21) & 18 (0.21) & 18 (0.21) & 18 (0.21) & 25 (0.32) & 26 (0.32) & 25 (0.32) \\
4 & 18 (0.21) & 18 (0.21) & 18 (0.21) & 18 (0.21) & 18 (0.21) & 18 (0.21) & 18 (0.21) & 26 (0.33) & 27 (0.35) & 27 (0.35) \\
5 & 18 (0.21) & 18 (0.21) & 18 (0.21) & 18 (0.21) & 18 (0.21) & 18 (0.21) & 18 (0.21) & 27 (0.34) & 29 (0.38) & 28 (0.37) \\
\hline
\end{tabular}}
\medskip
\caption{Number of Multigrid iterations and asymptotic convergence factors when $\omega=1$ and $\rho_1=1$.}
\label{table-omega1-mg}
\end{center}
\end{table}
}

\begin{table}[!ht]
\begin{center}
{\def\arraystretch{1.2}
\begin{tabular}{|r|r|c|c|c|c|c|c|c|c|c|c|}
\hline
  &  & \multicolumn{10}{c|}{$\rho_2$} \\
\cline{3-12}
$\ell$ &  $N$ & {\tt 0} & {\tt 1e-8} & {\tt 1e-6} & {\tt 1e-4} & {\tt 1e-2} & {\tt 1e-0} & {\tt 1e+2} & {\tt 1e+4} & {\tt 1e+6} & {\tt 1e+8} \\
\hline
1 &       729 &  9 &  9 &  9 &  9 &  9 &  9 &  9 &  8 &  9 &  9 \\
2 &     4,913 & 10 & 10 & 10 & 10 & 10 & 10 & 10 & 11 & 11 & 11 \\
3 &    35,937 & 10 & 10 & 10 & 10 & 10 & 10 & 10 & 12 & 12 & 12 \\
4 &   274,625 & 10 & 10 & 10 & 10 & 10 & 10 & 10 & 12 & 13 & 12 \\
5 & 2,146,689 & 10 & 10 & 10 & 10 & 10 & 10 & 10 & 12 & 13 & 13 \\
\hline
\end{tabular}}
\medskip
\caption{Number of MG-CG iterations when $\omega=1$ and $\rho_1=1$.}
\label{table-omega1-mg-cg}
\end{center}
\end{table}

%%%%%%%%%%%%%%%%%%%%%%%%%%%%%%%%%%%%%%%%%%%%%%%%%%%%%%%%%%%%%%%%%%%%%%%%%%%%%%%%
%  Numerical experiments - Constant rho
%%%%%%%%%%%%%%%%%%%%%%%%%%%%%%%%%%%%%%%%%%%%%%%%%%%%%%%%%%%%%%%%%%%%%%%%%%%%%%%%
\subsection{The case of constant $\rho$}
Next, we consider the case when the mass term coefficient is a constant.
As in the previous section, we first perform a parameter study to
determine an appropriate scaling of $\rho$ when $\omega_2$ is fixed to
be one.
The results are presented in Table \ref{table-rhoconst-gs-l34}, and in many
respects are similar to those from Table \ref{table-omega1-gs-l34}.
For example, the number of GS-GC iterations doubles from one level to
the next, though the actual numbers are several times larger than
those in Table \ref{table-omega1-gs-l34}.

{\Small
\begin{table}[!ht]
\begin{center}
{\def\arraystretch{1.2}
\begin{tabular}{|r|c||c|c|c|c|c|c|c|c|c|c|}
\hline
\multicolumn{2}{|c||}{$$} & \multicolumn{10}{c|}{$\rho$} \\
\cline{3-12}
\multicolumn{2}{|c||}{$\omega_2=1$} & {\tt 0} & {\tt 1e-8} & {\tt 1e-6} & {\tt 1e-4} & {\tt 1e-2} & {\tt 1e-0} & {\tt 1e+2} & {\tt 1e+4} & {\tt 1e+6} & {\tt 1e+8} \\
\hline
\hline
\multicolumn{12}{|c|}{$\ell=3$, $h^2 \approx$ {\tt 1e-3}, $N=35,937$}  \\
\hline
& {\tt 1e-8} & 117 & {\bf 173} & 90 & 60 & 59 & 53 & 35 & 15 & 15 & 15 \\
& {\tt 1e-6} & 108 & 108 & {\bf 107} & 82 & 58 & 53 & 35 & 15 & 15 & 15 \\
& {\tt 1e-4} & 97 & 97 & 97 & {\bf 97} & 75 & 52 & 35 & 15 & 15 & 15 \\
& {\tt 1e-2} & 87 & 87 & 87 & 87 & {\bf 87} & 66 & 34 & 15 & 15 & 15 \\
$\omega_1$
& {\tt 1e-0} & 62 & 62 & 62 & 62 & 62 & {\bf 62} & 46 & 10 & 15 & 15 \\
& {\tt 1e+2} & 74 & 74 & 74 & 74 & 74 & 74 & {\bf 73} & 48 & 12 & 15 \\
& {\tt 1e+4} & 68 & 68 & 68 & 68 & 68 & 68 & 68 & {\bf 69} & 48 & 12 \\
& {\tt 1e+6} & 63 & 63 & 63 & 63 & 63 & 63 & 63 & 64 & {\bf 69} & 48 \\
& {\tt 1e+8} & 59 & 59 & 59 & 59 & 59 & 59 & 60 & 59 & 65 & {\bf 69} \\
\hline
\hline
\multicolumn{12}{|c|}{$\ell=4$, $h^2 \approx$ {\tt 2.5e-4}, $N=274,625$}  \\
\hline
& {\tt 1e-8} & 348 & {\bf 347} & 178 & 107 & 102 & 90 & 62 & 15 & 15 & 15 \\
& {\tt 1e-6} & 212 & 222 & {\bf 211} & 163 & 102 & 90 & 62 & 15 & 15 & 15 \\
& {\tt 1e-4} & 193 & 193 & 193 & {\bf 192} & 150 & 91 & 62 & 15 & 15 & 15 \\
& {\tt 1e-2} & 169 & 169 & 169 & 169 & {\bf 168} & 128 & 61 & 14 & 15 & 15 \\
$\omega_1$
& {\tt 1e-0} & 120 & 120 & 120 & 120 & 120 & {\bf 120} & 85 & 13 & 14 & 15 \\
& {\tt 1e+2} & 141 & 141 & 141 & 141 & 141 & 141 & {\bf 140} & 94 & 13 & 15 \\
& {\tt 1e+4} & 132 & 132 & 132 & 132 & 132 & 132 & 132 & {\bf 132} & 94 & 14 \\
& {\tt 1e+6} & 124 & 124 & 124 & 124 & 124 & 124 & 124 & 123 & {\bf 132} & 94 \\
& {\tt 1e+8} & 113 & 113 & 113 & 113 & 113 & 113 & 113 & 112 & 123 & {\bf 132} \\
\hline
\end{tabular}}
\medskip
\caption{Number of GS-CG iterations when $\omega_2=1$ and $\rho$ is a constant.}
\label{table-rhoconst-gs-l34}
\end{center}
\end{table}
}

Examining the results in Table \ref{table-rhoconst-gs-l34}, we can
conclude that the most challenging problems occur when $\rho$ and
$\omega_1$ are of the same magnitude.
Therefore, we restrict the experiments in this section to the case
$\omega_2=1$, $\rho=\omega_1$.

\begin{table}[!ht]
\begin{center}
{\def\arraystretch{1.2}
\begin{tabular}{|r|r|c|c|c|c|c|c|c|c|c|}
\hline
  &  & \multicolumn{9}{c|}{$\omega_1$} \\
\cline{3-11}
$\ell$ &  $N$ & {\tt 1e-8} & {\tt 1e-6} & {\tt 1e-4} & {\tt 1e-2} & {\tt 1e-0} & {\tt 1e+2} & {\tt 1e+4} & {\tt 1e+6} & {\tt 1e+8} \\
\hline
1 &       729 & 30 & 25 & 23 & 21 & 18 & 19 & 19 & 19 & 19 \\
2 &     4,913 & 87 & 55 & 51 & 45 & 36 & 40 & 39 & 39 & 39 \\
3 &    35,937 & 173 & 107 & 97 & 87 & 62 & 73 & 69 & 69 & 69 \\
4 &   274,625 & 347 & 211 & 192 & 168 & 120 & 140 & 132 & 132 & 132 \\
\hline
\end{tabular}}
\medskip
\caption{Number of GS-CG iterations when $\omega_2=1$ and $\rho=\omega_1$.}
\label{table-rhoconst-gs}
\end{center}
\end{table}

The results for GS-CG are shown in Table \ref{table-rhoconst-gs}.
Clearly, the problem of hard enclosure, when $\omega_1$ is small, is
much more challenging than the case of large $\omega_1$.
In contrast to Table \ref{table-omega1-gs}, the number of iterations increases
significantly with the magnitude of the jump.
This is due to the fact that the condition number is
proportional to $\mathcal{J}(\omega)$, see Theorem \ref{th-gs} and the
discussion after Theorem 2.1 in \cite{Xu.J;Zhu.Y2008}.

\begin{table}[!ht]
\begin{center}
{\def\arraystretch{1.2}
\begin{tabular}{|r|r|c|c|c|c|c|c|c|c|c|}
\hline
  &  & \multicolumn{9}{c|}{$\omega_1$} \\
\cline{3-11}
$\ell$ &  $N$ & {\tt 1e-8} & {\tt 1e-6} & {\tt 1e-4} & {\tt 1e-2} & {\tt 1e-0} & {\tt 1e+2} & {\tt 1e+4} & {\tt 1e+6} & {\tt 1e+8} \\
\hline
1 &       729 & 21 & 22 & 22 & 22 & 20 & 20 & 20 & 20 & 20 \\
2 &     4,913 & 34 & 34 & 34 & 33 & 27 & 29 & 28 & 28 & 28 \\
3 &    35,937 & 41 & 41 & 41 & 40 & 31 & 33 & 32 & 32 & 32 \\
4 &   274,625 & 46 & 46 & 47 & 44 & 33 & 35 & 35 & 35 & 35 \\
5 & 2,146,689 & 51 & 51 & 52 & 48 & 35 & 38 & 38 & 37 & 38 \\
\hline
\end{tabular}}
\medskip
\caption{Number of BPX-CG iterations when $\omega_2=1$ and $\rho=\omega_1$.}
\label{table-rhoconst-bpx-cg}
\end{center}
\end{table}

To a lesser extend, this trend is present in the results with BPX preconditioning
reported in Table \ref{table-rhoconst-bpx-cg}.
Even though the increase in the number of iteration due to the jump in $\omega$ is
not as large as for GS-CG, the influence of $\mathcal{J}(\omega)$ on the condition
number can be observed if we plot the convergence history of the PCG iterations.
Such a plot is presented in Figure \ref{fig-bpx-cg-conv}, where one can clearly see
that when $\omega_1=10^{-8}$, PCG needs several extra iterations to resolve the
eigenvector corresponding to the isolated minimal eigenvalue, cf. Figure 3 in \cite{Xu.J;Zhu.Y2008}.

\begin{figure}
\centerline{\includegraphics[height=3in]{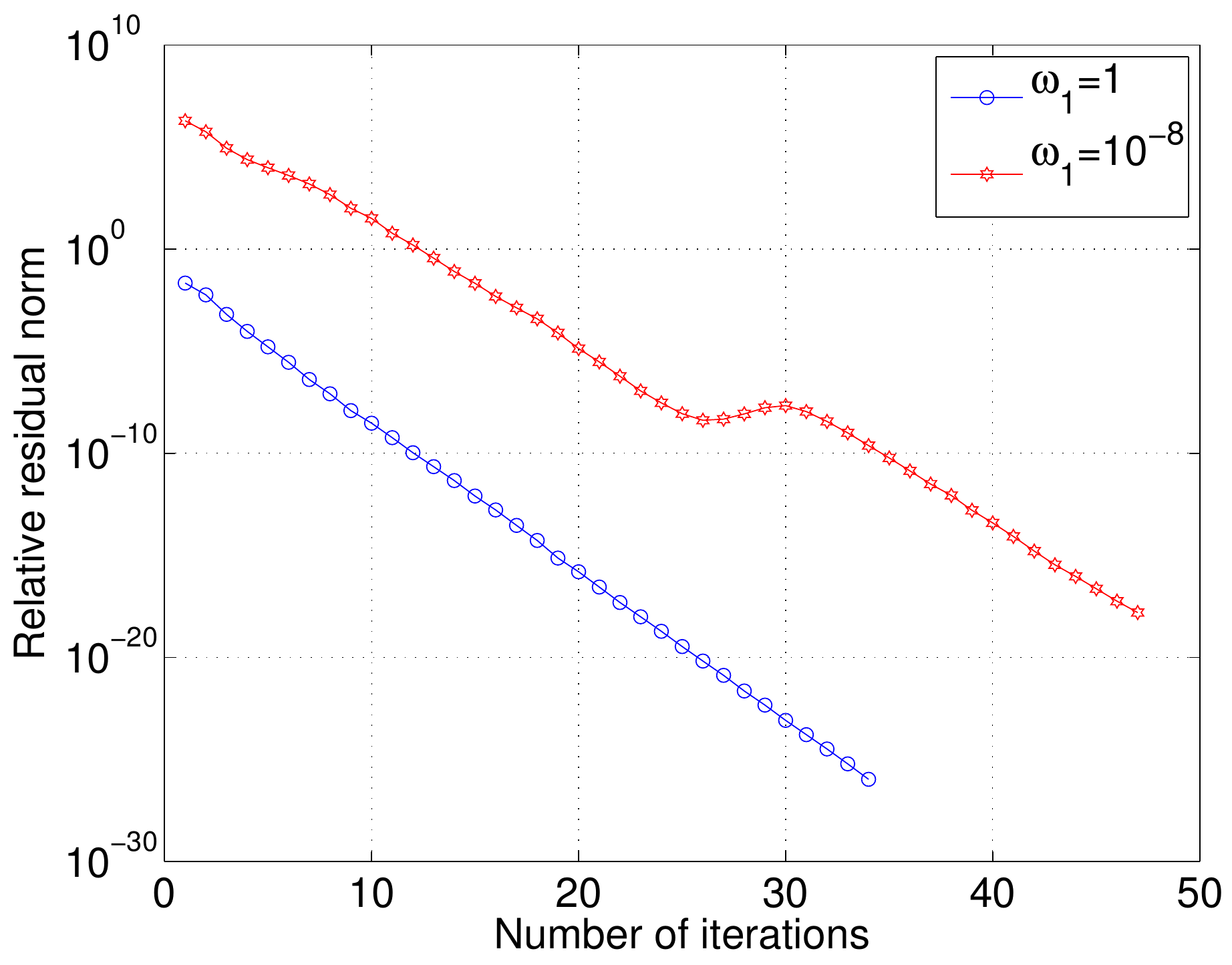}}
\caption{Convergence history for BPX-CG when $\omega_2=1$, $\rho=\omega_1$ and
$\omega_1 \in \{1, 10^{-8}\}$. Problem size $N=274,625$.}
\label{fig-bpx-cg-conv}
\end{figure}

\begin{table}[!ht]
\begin{center}
{\def\arraystretch{1.2}
\begin{tabular}{|r|r|c|c|c|c|c|c|}
\hline
  & \multicolumn{7}{c|}{$\omega_1$} \\
\cline{1-8}
$\ell$ & {\tt 1e-4} & {\tt 1e-2} & {\tt 1e-0} & {\tt 1e+2} & {\tt 1e+4} & {\tt 1e+6} & {\tt 1e+8} \\
\hline
1 &  41 (0.61) & 38 (0.55) & 16 (0.17) & 18 (0.20) & 18 (0.20) & 18 (0.20) & 18 (0.20) \\
2 & 100 (0.82) & 69 (0.74) & 18 (0.20) & 20 (0.24) & 20 (0.24) & 19 (0.24) & 19 (0.24) \\
3 & 216 (0.93) & 100 (0.81) & 18 (0.21) & 21 (0.26) & 21 (0.26) & 21 (0.27) & 21 (0.26) \\
4 & 440 (0.97) & 124 (0.85) & 18 (0.21) & 22 (0.29) & 22 (0.29) & 22 (0.29) & 22 (0.29) \\
5 & 843 (0.98) & 140 (0.87) & 18 (0.21) & 23 (0.31) & 23 (0.31) & 23 (0.31) & 23 (0.31) \\
\hline
\end{tabular}}
\medskip
\caption{Number of Multigrid iterations and asymptotic convergence factors when $\omega_2=1$ and $\rho=\omega_1$.}
\label{table-rhoconst-mg}
\end{center}
\end{table}

In the previous section we observed that Multigrid has asymptotic
convergence factor independent of the jumps in $\rho$ (see Table
\ref{table-omega1-mg}).
This is no longer true when $\omega$ is not a constant, as
demonstrated in Table \ref{table-rhoconst-mg}.
Indeed, the condition number of the Multigrid preconditioned system is
bounded by $\min \{\mathcal{J}(\omega), h^{-1}\}$, so when the jump is
large enough (as in the leftmost column) the iterations double with
each refinement level.

\begin{table}[!ht]
\begin{center}
{\def\arraystretch{1.2}
\begin{tabular}{|r|r|c|c|c|c|c|c|c|c|c|}
\hline
  &  & \multicolumn{9}{c|}{$\omega_1$} \\
\cline{3-11}
$\ell$ &  $N$ & {\tt 1e-8} & {\tt 1e-6} & {\tt 1e-4} & {\tt 1e-2} & {\tt 1e-0} & {\tt 1e+2} & {\tt 1e+4} & {\tt 1e+6} & {\tt 1e+8} \\
\hline
1 &       729 & 10 & 10 & 10 & 10 &  9 &  9 &  9 &  9 &  9 \\
2 &     4,913 & 13 & 13 & 13 & 13 & 10 & 11 & 11 & 11 & 11 \\
3 &    35,937 & 14 & 14 & 14 & 14 & 10 & 11 & 11 & 11 & 11 \\
4 &   274,625 & 15 & 15 & 15 & 15 & 10 & 11 & 11 & 11 & 11 \\
5 & 2,146,689 & 16 & 16 & 16 & 15 & 10 & 12 & 12 & 12 & 12 \\
\hline
\end{tabular}}
\medskip
\caption{Number of MG-CG iterations when $\omega_2=1$ and $\rho=\omega_1$.}
\label{table-rhoconst-mg-cg}
\end{center}
\end{table}

Using Multigrid as a preconditioner resolves this problem, since there are only
finite number of small eigenvalues corresponding to the jump in $\omega$.
The results in Table \ref{table-rhoconst-mg-cg} demonstrate a nearly optimal
convergence with respect to the mesh size.

%%%%%%%%%%%%%%%%%%%%%%%%%%%%%%%%%%%%%%%%%%%%%%%%%%%%%%%%%%%%%%%%%%%%%%%%%%%%%%%%
%  Discontinuous omega and rho
%%%%%%%%%%%%%%%%%%%%%%%%%%%%%%%%%%%%%%%%%%%%%%%%%%%%%%%%%%%%%%%%%%%%%%%%%%%%%%%%
\subsection{The case of discontinuous $\omega$ and $\rho$}\label{sec:num3}
In this section we present a numerical investigation of the general
case when both $\omega$ and $\rho$ are discontinuous.
Note that the theory developed in this paper can be applied only if we
can construct an interpolation operator which is stable in both the
$\rho$-weighted and the $\omega$-weighted $L^{2}$-inner products.
This is the case, for example if $\omega_1 \leq \omega_2$ and $\rho_1 \leq \rho_2$.

{\Small
\begin{table}[!ht]
\begin{center}
{\def\arraystretch{1.2}
\begin{tabular}{|r|c|c|c|c|c|c|c|c|c|c|}
\hline
\multicolumn{2}{|c|}{} & \multicolumn{9}{c|}{$\omega_1/\omega_2$} \\
\cline{3-11}
\multicolumn{2}{|c|}{} & {\tt 1e-8} & {\tt 1e-6} & {\tt 1e-4} & {\tt 1e-2} & {\tt 1e-0} & {\tt 1e+2} & {\tt 1e+4} & {\tt 1e+6} & {\tt 1e+8} \\
\hline
& {\tt 1e-8} & {\bf 169} & 170 & 170 & 164 & 133 & 141 & 139 & 139 & 113 \\
& {\tt 1e-6} & {\bf 193} & 194 & 191 & 169 & 133 & 141 & 139 & 139 & 113 \\
& {\tt 1e-4} & {\bf 214} & 209 & 193 & 169 & 133 & 141 & 139 & 139 & 113 \\
& {\tt 1e-2} & {\bf 344} & 213 & 193 & 169 & 132 & 141 & 138 & 138 & 122 \\
$\rho_1/\rho_2$
& {\tt 1e-0} & {\bf 347} & 222 & 193 & 169 & 120 & 141 & 132 & 132 & 132 \\
& {\tt 1e+2} & {\bf 268} & 221 & 193 & 169 & 120 & 141 & 132 & 124 & 133 \\
& {\tt 1e+4} & {\bf 111} & 164 & 192 & 169 & 120 & 141 & 132 & 124 & 133 \\
& {\tt 1e+6} & {\bf 108} & 104 & 151 & 168 & 120 & 141 & 132 & 124 & 133 \\
& {\tt 1e+8} & {\bf 101} & 101 & 100 & 131 & 120 & 141 & 132 & 124 & 132 \\
\hline
\end{tabular}}
\medskip
\caption{Number of GS-CG iterations when $\omega_2=1$, while $\omega_1$, $\rho_1$ and $\rho_2$ are allowed to vary. Each cell in the table represents a maximum over a range of values for $\rho$. Problem size $N=274,625$.}
\label{table-gen-gs-l4}
\end{center}
\end{table}
}

In Table \ref{table-gen-gs-l4} we show the results of a parameter
study based on Gauss-Seidel preconditioning.
We emphasize that each cell in this table represents a maximum over several
possible values for $\rho$, which result in a jump of the same
magnitude $\rho_1/\rho_2$.
Clearly, the difficulty of the problem is determined mostly by the
jump in $\omega$, so we choose to concentrate on the most challenging
case $\omega_1=10^{-8}.$

The results of using for BPX and Multigrid V-cycle preconditioners for
this choice of $\omega$ are shown in Table \ref{table-gen-bpx-cg} and
Table \ref{table-gen-mg-cg} respectively.
They indicate that when $\rho_1 \leq \rho_2$, the PCG behavior is
generally similar to the case when $\rho$ is a constant.
This is not surprising, since as we mentioned earlier, our convergence
theory can be applied in this special case.
When $\rho_1 > \rho_2$, the convergence deteriorates, though not
significantly.

\begin{table}[!ht]
\begin{center}
{\def\arraystretch{1.2}
\begin{tabular}{|r|r|c|c|c|c|c|c|c|c|c|}
\hline
  &  & \multicolumn{9}{c|}{$\rho_1/\rho_2$} \\
\cline{3-11}
$\ell$ &  $N$ & {\tt 1e-8} & {\tt 1e-6} & {\tt 1e-4} & {\tt 1e-2} & {\tt 1e-0} & {\tt 1e+2} & {\tt 1e+4} & {\tt 1e+6} & {\tt 1e+8} \\
\hline
1 &       729 & 20 & 20 & 20 & 21 & 21 & 21 & 21 & 21 & 21 \\
2 &     4,913 & 32 & 33 & 33 & 33 & 34 & 32 & 32 & 32 & 32 \\
3 &    35,937 & 39 & 40 & 40 & 40 & 41 & 42 & 42 & 42 & 42 \\
4 &   274,625 & 44 & 45 & 45 & 46 & 46 & 48 & 49 & 49 & 49 \\
%% 5 & 2,146,689 & 16 & 16 &    &    & 10 &    &    &    &    \\
\hline
\end{tabular}}
\medskip
\caption{Number of BPX-CG iterations when $\omega_1=10^{-8}$ and $\omega_2=1$. Each cell in the table represents a maximum over a range of values for $\rho$.}
\label{table-gen-bpx-cg}
\end{center}
\end{table}

\begin{table}[!ht]
\begin{center}
{\def\arraystretch{1.2}
\begin{tabular}{|r|r|c|c|c|c|c|c|c|c|c|}
\hline
  &  & \multicolumn{9}{c|}{$\rho_1/\rho_2$} \\
\cline{3-11}
$\ell$ &  $N$ & {\tt 1e-8} & {\tt 1e-6} & {\tt 1e-4} & {\tt 1e-2} & {\tt 1e-0} & {\tt 1e+2} & {\tt 1e+4} & {\tt 1e+6} & {\tt 1e+8} \\
\hline
1 &       729 & 10 & 10 & 10 & 10 & 10 & 10 & 10 & 10 & 10 \\
2 &     4,913 & 13 & 13 & 13 & 13 & 13 & 13 & 13 & 13 & 13 \\
3 &    35,937 & 14 & 14 & 14 & 14 & 14 & 15 & 15 & 15 & 15 \\
4 &   274,625 & 14 & 15 & 15 & 15 & 15 & 17 & 17 & 17 & 17 \\
%% 5 & 2,146,689 & 16 & 16 &    &    & 10 &    &    &    &    \\
\hline
\end{tabular}}
\medskip
\caption{Number of MG-CG iterations when $\omega_1=10^{-8}$ and $\omega_2=1$.  Each cell in the table represents a maximum over a range of values for $\rho$.}
\label{table-gen-mg-cg}
\end{center}
\end{table}

To further investigate the effect of adding jumps in $\rho$, when
$\omega$ is already discontinuous we consider a test problem in two
dimensions.
We start with the coarse triangulation shown in Figure
\ref{fig-square-36-rand} and randomly assign each coarse triangle to
one of two possible subdomains.  The mesh is then refined $\ell $
times.

\begin{figure}
\centerline{\includegraphics[height=2in]{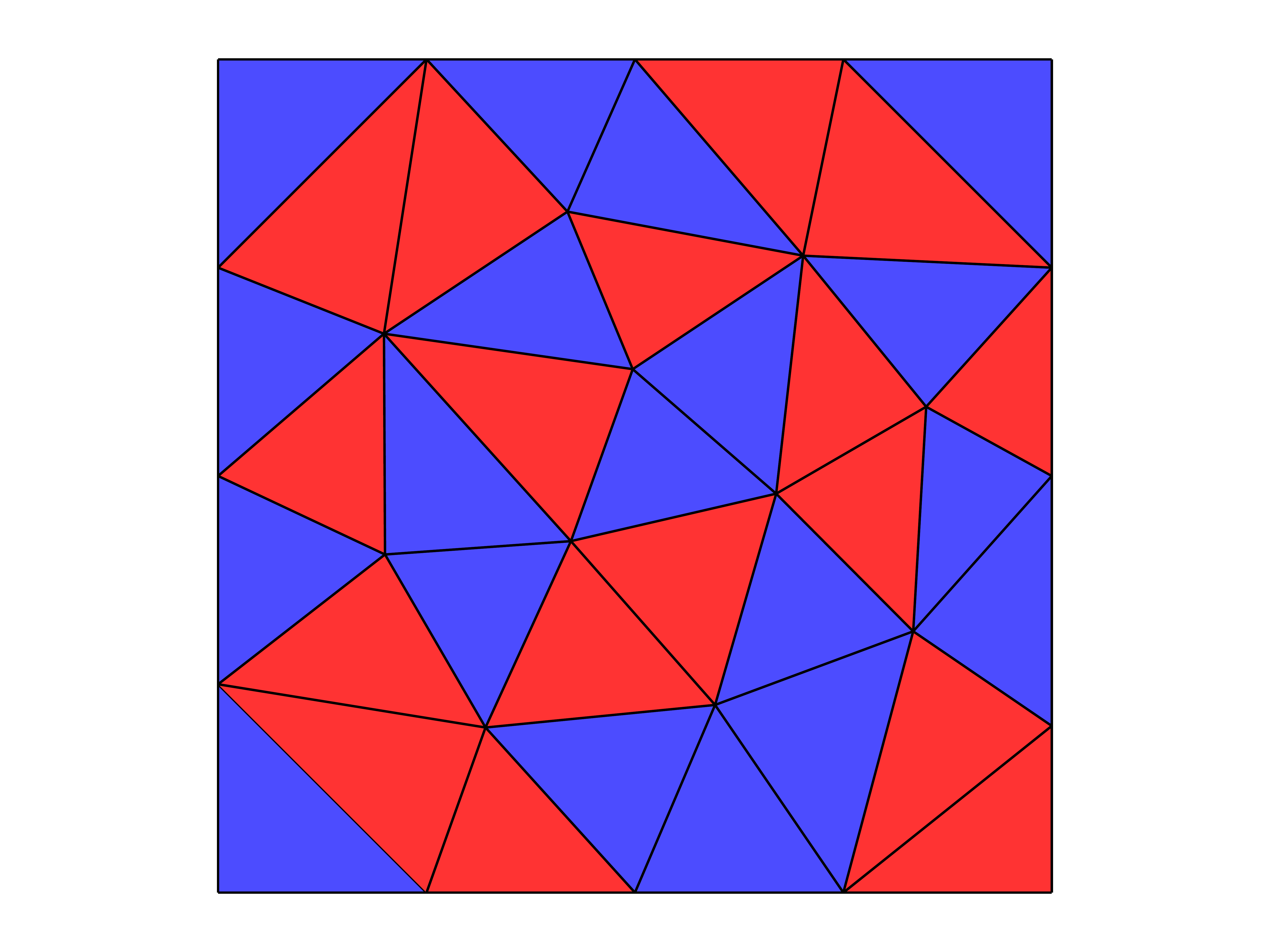}}
\caption{Coarse triangulation ($\ell=0$) and the two material subdomains for the two-dimensional test problem.}
\label{fig-square-36-rand}
\end{figure}

We focus on the case $\omega_1=10^{-8}$ and $\omega_2=1$ and allow
$\rho_1$ and $\rho_2$ to vary as in the previous experiments.
The results for BPX and Multigrid preconditioners are shown in Table
\ref{table-2d-gen-bpx-cg} and Table \ref{table-2d-gen-mg-cg}.
They appear to indicate that adding jumps in $\rho$ can lead to a
significant deterioration in the convergence of this problem.
The approximate solution corresponding to one of the most challenging
cases is plotted in Figure \ref{fig-square-36-sol}.

\begin{table}[!ht]
\begin{center}
{\def\arraystretch{1.2}
\begin{tabular}{|r|r|c|c|c|c|c|c|c|c|c|}
\hline
  &  & \multicolumn{9}{c|}{$\rho_1/\rho_2$} \\
\cline{3-11}
$\ell$ &  $N$ & {\tt 1e-8} & {\tt 1e-6} & {\tt 1e-4} & {\tt 1e-2} & {\tt 1e-0} & {\tt 1e+2} & {\tt 1e+4} & {\tt 1e+6} & {\tt 1e+8} \\
\hline
4 &     4,737 & 49 & 50 & 51 & 53 & 56 & 63 & 64 & 64 & 64 \\
5 &    18,689 & 57 & 58 & 59 & 62 & 66 & 78 & 79 & 79 & 79 \\
6 &    74,241 & 63 & 67 & 67 & 74 & 77 & 93 & 95 & 95 & 95 \\
7 &   295,937 & 73 & 76 & 76 & 87 & 93 & 109& 122& 123& 123\\
8 & 1,181,697 & 81 & 83 & 83 & 100& 110& 125& 164& 164& 164\\
9 & 4,722,689 & 88 & 90 & 90 & 114& 127& 141& 209& 211& 211\\
\hline
\end{tabular}}
\medskip
\caption{Two-dimensional test problem: Number of BPX-CG iterations when $\omega_1=10^{-8}$ and $\omega_2=1$.  Each cell in the table represents a maximum over a range of values for $\rho$.}
\label{table-2d-gen-bpx-cg}
\end{center}
\end{table}

\begin{table}[!ht]
\begin{center}
{\def\arraystretch{1.2}
\begin{tabular}{|r|r|c|c|c|c|c|c|c|c|c|}
\hline
  &  & \multicolumn{9}{c|}{$\rho_1/\rho_2$} \\
\cline{3-11}
$\ell$ &  $N$ & {\tt 1e-8} & {\tt 1e-6} & {\tt 1e-4} & {\tt 1e-2} & {\tt 1e-0} & {\tt 1e+2} & {\tt 1e+4} & {\tt 1e+6} & {\tt 1e+8} \\
\hline
4 &     4,737 & 18 & 18 & 18 & 19 & 20 & 23 & 23 & 23 & 23 \\
5 &    18,689 & 19 & 21 & 21 & 21 & 22 & 26 & 26 & 26 & 26 \\
6 &    74,241 & 21 & 23 & 23 & 23 & 25 & 29 & 30 & 30 & 30 \\
7 &   295,937 & 23 & 25 & 25 & 25 & 27 & 32 & 40 & 40 & 40 \\
8 & 1,181,697 & 26 & 26 & 26 & 27 & 30 & 36 & 52 & 53 & 53 \\
9 & 4,722,689 & 28 & 28 & 28 & 30 & 32 & 40 & 65 & 66 & 66 \\
\hline
\end{tabular}}
\medskip
\caption{Two-dimensional test problem: Number of MG-CG iterations when $\omega_1=10^{-8}$ and $\omega_2=1$.  Each cell in the table represents a maximum over a range of values for $\rho$.}
\label{table-2d-gen-mg-cg}
\end{center}
\end{table}

\begin{figure}
\centerline{
\includegraphics[height=2in]{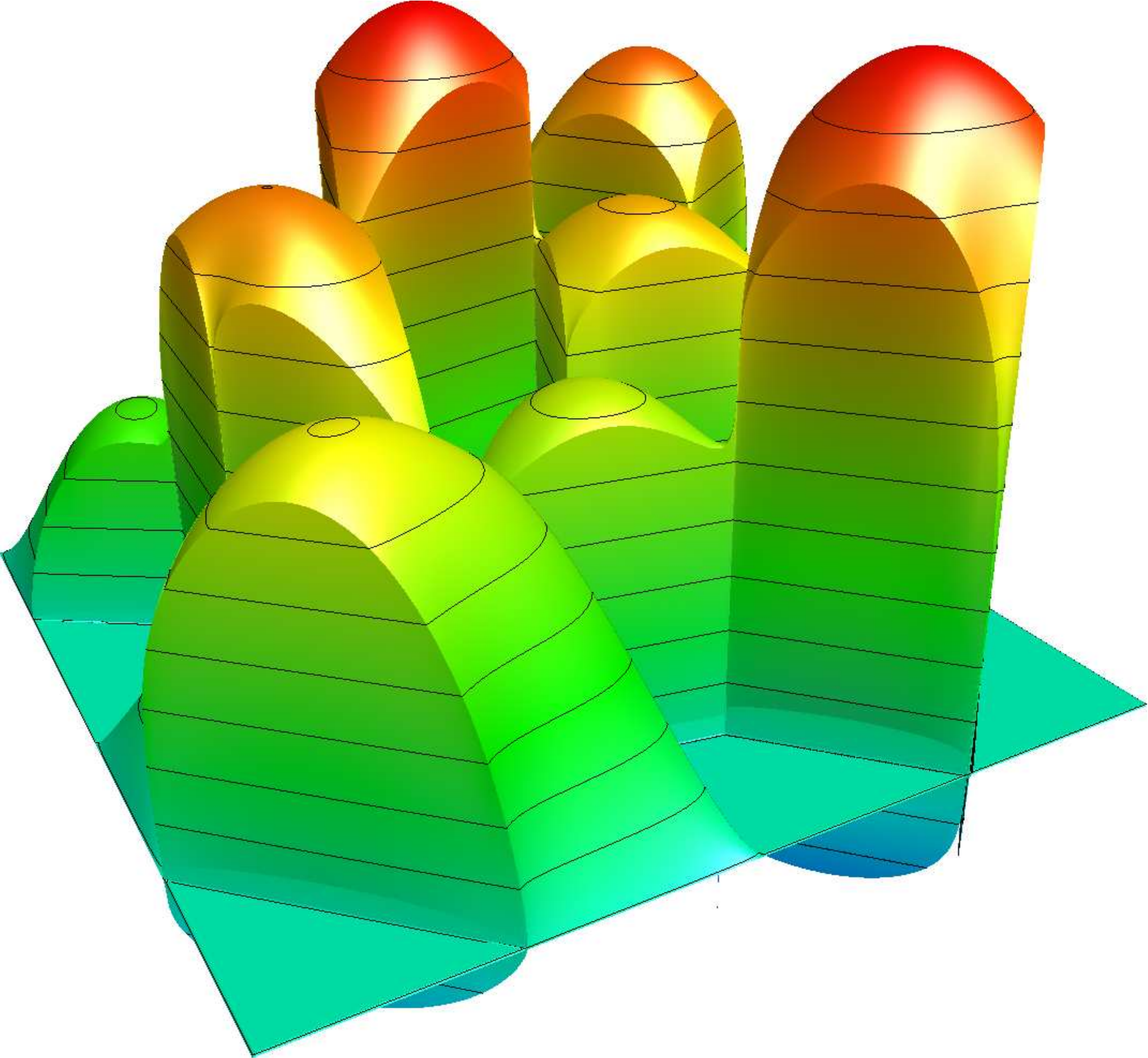}
}
\caption{Approximate solutions corresponding to $\omega_1=10^{-8}$, $\omega_2=1$, $\rho_1=10^4$ and $\rho_2=1$.}
\label{fig-square-36-sol}
\end{figure}

\end{document}